\documentclass[11pt]{article} %
\usepackage[a4paper, left=25mm, top=25mm, right=25mm, bottom=25mm]{geometry}
\usepackage{marginnote}

\usepackage{booktabs}

\usepackage[english]{babel}
\usepackage[T1]{fontenc}

\usepackage{amssymb,amsfonts,amsmath,amsthm}

\usepackage{mathrsfs}
\usepackage{stmaryrd}

\usepackage{stackengine}
\usepackage{mathtools}

\usepackage[hidelinks,unicode=true]{hyperref}
\usepackage{color}

\usepackage{dsfont}
\usepackage{comment}
\usepackage{lastpage}
\usepackage{appendix}
\usepackage{esvect}
\usepackage{esint}

\usepackage{float}

\usepackage{tikz}
\usetikzlibrary{decorations.pathmorphing,patterns}
\usepackage{tkz-graph}
\usepackage{xcolor}
\usepackage{pgfplots}
\pgfplotsset{width=7cm,compat=1.5.1}

\usetikzlibrary{arrows.meta,calc,decorations.pathreplacing,intersections}
\usetikzlibrary{decorations.markings, positioning}

\renewcommand{\paragraph}[1]{$\empty$\\\textit{#1}.\hspace{0.5em}}

\numberwithin{equation}{section}

\usepackage{enumitem}

\newenvironment{itempuce}%
{\begin{itemize}[
	label=$-$, %
	leftmargin=*, %
]}{\end{itemize}}

\theoremstyle{definition}
\newtheorem{defn}{Definition}

\newtheorem{thm}{Theorem}
\newtheorem{lem}[thm]{Lemma}
\newtheorem{prop}[thm]{Proposition}
\newtheorem{cor}[thm]{Corollary}

\theoremstyle{remark}
\newtheorem{rmq}[defn]{Remark}

\newcommand{\fa}{\forall \,}

\renewcommand{\and}{\qquad\mathrm{and}\qquad}

\renewcommand{\leq}{\,\leqslant\,}
\renewcommand{\geq}{\,\geqslant\,}

\newcommand{\Lra}{\Longrightarrow}

\newcommand{\lra}{\longrightarrow}

\newcommand{\xra}[2]{\;\;\xrightarrow[#2]{#1}\;\;}
\newcommand{\lms}{\longmapsto}

\newcommand{\ps}[1]{\left\langle #1 \right\rangle}

\newcommand{\event}[1]{\left\{#1\right\}}
\newcommand{\vabs}[1]{\left|#1\right|}
\newcommand{\parent}[1]{\left(#1\right)}

\newcommand{\norme}[2]{\left\|#1\right\|_{#2}}

\newcommand{\dist}{\dist}

\newcommand{\mtext}[1]{\mbox{\rm #1}}

\renewcommand{\d}{\mathrm{d}}

\renewcommand{\H}{\mathrm{H}}

\renewcommand{\P}{\mathrm{P}}

\newcommand{\R}{\mathrm{R}}
\renewcommand{\S}{\mathrm{S}}

\newcommand{\1}{\mathds{1}}

\newcommand{\EE}{\mathbb{E}}

\newcommand{\PP}{\mathbb{P}}

\newcommand{\RR}{\mathbb{R}}

\newcommand{\AAA}{\mathcal{A}}

\newcommand{\DDD}{\mathscr{D}}
\newcommand{\EEE}{\mathcal{E}}
\newcommand{\FFF}{\mathcal{F}}
\newcommand{\GGG}{\mathcal{G}}

\newcommand{\KKK}{\mathcal{K}}
\newcommand{\LLL}{\mathcal{L}}
\newcommand{\MMM}{\mathcal{M}}
\newcommand{\NNN}{\mathcal{N}}

\newcommand{\PPP}{\mathcal{P}}

\newcommand{\TTT}{\mathcal{T}}

\newcommand{\XXX}{\mathcal{X}}

\newcommand{\Law}{\mtext{Law}}
\renewcommand{\dist}{\mtext{dist}}

\newcommand{\Wasserstein}{\mtext{Wasserstein}}

\newcommand{\Lip}{\mtext{Lip}}

\makeatletter

\def\@author{}
\def\@address{}
\def\@email{}
\def\@keywords{}
\def\@subjclass{}
\def\@subjclassyear{}

\renewcommand{\author}[1]{%
  \g@addto@macro\@author{#1\par}%
}

\newcommand{\address}[2][]{
  \if\relax\detokenize{#1}\relax
    \g@addto@macro\@address{#2\par}%
  \else
    \g@addto@macro\@address{\textsuperscript{#1}~#2\par}%
  \fi
}

\newcommand{\keywords}[1]{\def\@keywords{#1}}

\newcommand{\subjclass}[2][]{%
  \def\@subjclass{#2}%
  \def\@subjclassyear{#1}%
}

\renewcommand{\maketitle}{
\begin{center}
	\vspace*{2em}
	{\bfseries\LARGE \@title \par}
	\vskip 1.5em
	{\large \@author \par}
	\vskip 0.95em
	{\large \@date \par}
	\vspace*{1.5em}
\end{center}
\begingroup
\renewcommand\thefootnote{}
\long\def\@makefntext##1{\parindent=0pt\noindent##1}%
\footnotetext{%
	\ifx\@address\@empty\else%
		\@address%
	\fi
}
\endgroup
}

\newcommand{\keywordssubjclass}{
	\noindent
	\ifx\@keywords\@empty\else
		\textbf{Keywords.} \@keywords \par
	\fi
\medskip
	\noindent
	\ifx\@subjclass\@empty\else
		\textbf{Mathematics Subject Classification.} %
		\@subjclass\par
	\fi
\medskip
}

\makeatother
 
\usepackage{enumitem}

\title{Metastable confinement of particle systems near linearly stable
McKean--Vlasov equilibria}
\author{Samuel Chan-ashing\textsuperscript{1,2,3}, Quentin Cormier\textsuperscript{3}}
\address[1]{DMA, École normale supérieure, Université PSL, CNRS, 75005 Paris, France}
\address[2]{CEREMADE, Université Paris-Dauphine, PSL, CNRS}
\address[3]{CMAP, Ecole Polytechnique, Palaiseau, Paris, France}
\date{\today}

\keywords{McKean--Vlasov equation; Mean-field particle system; Metastability.}

\subjclass[2020]{Primary 60K35; Secondary 60H10, 35Q84, 45D05.}

\begin{document}

\maketitle

\begin{abstract}
We establish a general metastability principle showing that the local stability
of a McKean--Vlasov equilibrium implies exponentially long confinement of the
associated particle system. Consider interacting diffusions on $\RR^d$ with a
general smooth and bounded interaction kernel and a globally Lipschitz drift
that is dissipative at large distances. Assume the McKean--Vlasov equation has at
least one equilibrium which is locally exponentially stable. We prove that the
probability that the empirical measure leaves a sufficiently small fixed
Wasserstein-1 neighborhood of the selected equilibrium before time $T$ is
bounded by $C(1+T)e^{-cN}$, where $N$ is the number of particles and $C,c>0$ are
independent of $N$ and $T$. Consequently, confinement holds up to times
$e^{\gamma N}$, for every $0<\gamma<c$, with exponentially small exit
probability. The proof combines an explicit Volterra construction of the
linearized Fokker--Planck semigroup, a mild representation of empirical fluctuations, and exponential maximal estimates for stochastic convolutions. Our approach does not require a gradient flow structure or reversibility of the equations.
\end{abstract}

\keywordssubjclass

\tableofcontents

\section{Introduction}

A diffusion may admit a unique invariant probability measure and nevertheless converge very slowly to equilibrium. On observable time scales, its trajectories may remain trapped near locally stable states that do not describe the final equilibrium. This separation between asymptotic equilibrium and observable dynamics is one of the characteristic manifestations of \emph{metastability}. 
A canonical example is Kramers' small-noise diffusion in a multi-well potential \cite{Kra40,BH15}. Although its Gibbs measure describes the global equilibrium, a trajectory initialized in one well typically remains near a local minimum for an exponentially long time as the noise intensity tends to zero, before making a rare transition to another well.

An analogous situation occurs in large systems of interacting particles. Here, the number of particles \(N\) plays the role of the asymptotic parameter. 
For fixed $d\geq1$, consider the system
\begin{equation}\label{eq:spl}
\d X_t^{i,N}=\d B_t^i+\parent{b(X_t^{i,N})+\frac1N\sum_{j=1}^NF(X_t^{i,N},X_t^{j,N})}\d t,\qquad 1\leq i\leq N,
\end{equation}
where $B^1,\ldots,B^N$ are independent $d$-dimensional Brownian motions. 
For each fixed \(N\), the particle process is a finite-dimensional Markov diffusion. Under the assumptions considered below, it admits a unique invariant probability measure \(\pi^N\), and its law converges towards it as time tends to infinity. Although reflection-coupling methods provide quantitative exponential convergence estimates \cite{Ebe16}, the corresponding rates may depend on \(N\), and the time needed to approach \(\pi^N\) may grow rapidly with \(N\). To understand the dynamics before global equilibrium is reached, one is naturally led to consider the mean-field limit.

Classical propagation-of-chaos results \cite{Szn91} show that, on finite time intervals, the particle system is well approximated by its mean-field limit when \(N\) is large. Unlike the finite particle system, the limiting McKean--Vlasov equation may admit several invariant measures. Their coexistence and the associated phase transitions have been studied in various models, including double-well models by Tugaut \cite{Tug14} and interaction models on the torus by Carrillo, Gvalani, Pavliotis, and Schlichting \cite{CGPS20}. Locally stable invariant measures of the limiting equation are therefore natural candidates for the metastable states of the finite system.

However, the usual finite-time approximation estimates, whose constants grow exponentially in time, are typically informative only up to times of order \(\log N\). Under additional assumptions, the approximation can hold uniformly in time. Classical results under convexity assumptions include Malrieu \cite{Mal01}, while Lacker and Le Flem \cite{LL23} obtain sharp uniform-in-time estimates for particle marginals. Delarue and Tse \cite{DT25} develop an approach to uniform-in-time weak propagation of chaos on the torus based on the long-time behavior of the linearized Fokker--Planck equation. 
When several McKean--Vlasov equilibria coexist, uniform-in-time propagation of chaos cannot hold for all initial laws. Indeed, the McKean--Vlasov equation remains at whichever equilibrium it is initialized, whereas, for each fixed $N$, the particle system converges to the same invariant law regardless of its initial condition. A uniform-in-time approximation cannot therefore remain valid near every equilibrium. For the Curie--Weiss model, Journel and Le Bris \cite{JL25} obtain a uniform-in-time approximation by conditioning on the system not having left a prescribed region containing the selected equilibrium. Without this conditioning, the question is whether local stability can ensure that the empirical measure remains near a chosen equilibrium with high probability for times exponential in $N$.

Long-time confinement has already been established in several specific models. For mean-field plane rotators, Bertini, Giacomin, and Poquet \cite{BGP14} show that the empirical measure remains close to the manifold of synchronized profiles while the collective phase undergoes a Brownian motion on times of order \(N\). Building on their mild-formulation approach, Coppini \cite{Cop22} studies the stochastic Kuramoto model on suitable graph sequences and proves confinement near the set of stable stationary profiles over time horizons \(T_N=\exp(o(N))\).
In a different setting, Löcherbach and Monmarché \cite{LM22} establish metastable persistence and exit-time results for interacting neurons. For Wasserstein gradient flows with several stationary solutions, Monmarché and Reygner \cite{MR26} establish local exponential stability of the limiting equation through local nonlinear functional inequalities. For gradient mean-field diffusions with quadratic interaction and a multi-well confinement potential, Monmarché \cite{Mon25} obtains propagation of chaos on exponentially large time scales and an asymptotically exponential law for rescaled exit times. His approach uses modified dynamics whose invariant measure satisfies a log-Sobolev inequality uniform in the number of particles. These estimates concern the joint law of the particles and its marginals for gradient dynamics. Our estimates are complementary, since we control the trajectories of the empirical measure without requiring a gradient structure. See also Seo, Shin, Monmarché, and Choi \cite{SSMC26} for related results.
A complementary perspective is provided by the large-deviation theory of Dawson and Gärtner \cite{DG87,DG89}, which gives a framework for studying rare fluctuations of empirical measures and their relation to quasipotentials, describing the minimal action required to reach a given state from an attractor.

Our purpose here is to obtain a quantitative confinement estimate near a selected linearly stable McKean--Vlasov equilibrium. 
Under the stated regularity and confinement assumptions, we work on $\RR^d$ with a general bounded interaction kernel, without requiring a free-energy gradient-flow structure or reversibility. Starting sufficiently close to equilibrium, and under the initial moment conditions specified below, we show that the empirical measure remains in a small fixed neighborhood in $1$-Wasserstein distance for times exponential in $N$, with exponentially small exit probability. We illustrate this result with the Ornstein--Uhlenbeck--cosine model studied in \cite{Cor25}. 
For suitable interaction strengths, its limiting McKean--Vlasov equation admits several stable Gaussian equilibria, and our theorem yields exponentially long confinement near each of them. We also extend the result to the double-well model with linear interaction studied by Tugaut \cite{Tug14}, despite its cubic drift and unbounded interaction.

Our proof relies on the construction of a linearized semigroup, following the Volterra approach developed by Cormier \cite{Cor25} to study the local stability of McKean--Vlasov equilibria. This semigroup provides a mild representation of the empirical fluctuations around the selected equilibrium. The non-compact state space is handled through a separate control of the empirical second moment. This controls the contribution of the tails and allows us to approximate the Kantorovich class on a compact ball by finitely many smooth test functions. We then combine linear stability with estimates for the quadratic remainder and exponential maximal estimates for the stochastic convolutions.

\subsection{The model}

Fix \(d\geq1\) and \(N\geq1\), and consider \eqref{eq:spl} with empirical measure
\begin{equation}\label{eq:empirical-measure}
\mu_t^N:=\frac1N\sum_{j=1}^N\delta_{X_t^{j,N}}.
\end{equation}
We impose the following assumptions.

\begin{enumerate}[label=\textbf{(H\arabic*)}, ref=(H\arabic*), series=mes_hyps]
	\item \label{hyp:H1} The interaction $F$ belongs to $C_b^2(\RR^d\times\RR^d;\RR^d)$. Fix $L_F>0$ such that
	\[\norme{\nabla F}{\infty}+\norme{\nabla^2 F}{\infty}\leq L_F.\]
	The vector field $b\in C^2(\RR^d;\RR^d)$ is globally Lipschitz, and there are $R,\beta>0$ such that
	\[(x-y)\cdot\parent{b(x)-b(y)}\leq-\beta|x-y|^2,\qquad|x-y|\geq R.\]
\end{enumerate}

For $p=1,2$, $\mathcal P_p(\R^d)$ denotes the probability measures with finite $p$-th moment. For $\mu,\nu\in\PPP_1(\RR^d)$, the Kantorovich--Rubinstein formula reads
\begin{equation}\label{eq:W1}
\Wasserstein_1(\mu,\nu)=\sup_{\substack{g\in C^1(\RR^d),\ g(0)=0\\\norme{\nabla g}{\infty}\leq1}}\vabs{\ps{g,\mu-\nu}}.
\end{equation}

\paragraph{A unique invariant measure for the particle system}
Under \ref{hyp:H1}, the coefficients of \eqref{eq:spl} are globally Lipschitz, and the particle system is therefore globally well posed. Moreover, the drift $b$ is confining at infinity, while the interaction $F$ is bounded. Standard ergodic results for non-degenerate diffusions then imply that, for every fixed $N$, the $N$-particle process $X_t^N:=(X_t^{1,N},\ldots,X_t^{N,N})$ admits a unique invariant probability measure $\pi^N$, and that its law converges to $\pi^N$ as $t\to\infty$. This convergence can also be quantified. Applying the reflection coupling estimates of \cite[Corollary~2 and Lemma~1]{Ebe16} then gives constants $C_N\geq1$ and $\lambda_N>0$ such that, whenever $\Law(X_0^N)\in\PPP_1(\RR^{dN})$,
\[\Wasserstein_1\parent{\Law(X_t^N),\pi^N}\leq C_Ne^{-\lambda_Nt}\Wasserstein_1\parent{\Law(X_0^N),\pi^N},\qquad t\geq0,\]
where the Wasserstein distance is taken on $\RR^{dN}$ with its Euclidean norm. The constants can be chosen so that
\[C_N\leq Ce^{cN},\qquad \lambda_N^{-1}\leq Ce^{cN},\]
for some $C,c>0$ independent of $N$. For every fixed $\varepsilon\in(0,1)$, these estimates give an upper bound, exponential in $N$, on the time required to reduce the initial Wasserstein distance by a factor $\varepsilon$. They do not, however, determine the optimal convergence time. Our question is different: we study how long the empirical measure can remain close to a linearly stable equilibrium of the limiting McKean--Vlasov equation.

\paragraph{Mean-field limit and invariant measures}
Natural candidates for the metastable states of the particle system are provided by the stable invariant measures of its mean-field limit. For an initial probability measure \(\nu\), the limiting dynamics is described by the McKean--Vlasov equation
\begin{equation}\label{eq:mvl}
\d X_t^\nu=\d B_t+\parent{b(X_t^\nu)+\Phi(\mu_t^\nu)(X_t^\nu)}\d t,\qquad \mu_t^\nu:=\Law(X_t^\nu),\qquad \mu_0^\nu=\nu,
\end{equation}
where
\[\Phi(\mu)(x):=\int_{\RR^d}F(x,y)\,\mu(\d y).\]
For a bounded globally Lipschitz vector field \(\alpha:\RR^d\to\RR^d\), consider the diffusion
\begin{equation}\label{eq:mvll}
\d Y_t^\alpha=\d B_t+\parent{b(Y_t^\alpha)+\alpha(Y_t^\alpha)}\d t.
\end{equation}
Under \ref{hyp:H1}, this diffusion admits a unique invariant probability
measure, denoted by \(\nu_\infty^\alpha\). Suppose first that \(\nu_\infty\) is
invariant for the McKean--Vlasov equation~\eqref{eq:mvl} and define \(\alpha:=\Phi(\nu_\infty)\). When the equation is initialized with law \(\nu_\infty\), its law remains equal to \(\nu_\infty\), so the interaction field remains equal to \(\alpha\). Hence \(\nu_\infty\) is invariant for the diffusion \eqref{eq:mvll}, and therefore \(\nu_\infty=\nu_\infty^\alpha\). Conversely, suppose that
\begin{equation}\label{eq:general-fixed-point}
\alpha=\Phi(\nu_\infty^\alpha),
\end{equation}
and that \(Y_0^\alpha\sim\nu_\infty^\alpha\). Then \(Y^\alpha\) also solves \eqref{eq:mvl}. Therefore, \(\nu_\infty^\alpha\) is invariant for the McKean--Vlasov equation. Thus the invariant laws of \eqref{eq:mvl} are precisely the measures \(\nu_\infty^\alpha\) associated with solutions of the self-consistency relation \eqref{eq:general-fixed-point}.

Since our aim is to study the dynamics near an invariant law, we fix one such law in the following assumption.
\begin{enumerate}[resume*=mes_hyps]
	\item \label{hyp:H2} We fix an invariant probability measure
		$\nu_\infty$ of the McKean--Vlasov equation~\eqref{eq:mvl} and set $\alpha:=\Phi(\nu_\infty)$. No uniqueness is assumed.
\end{enumerate}
By the preceding characterization, \(\nu_\infty=\nu_\infty^\alpha\) and \(\alpha\) satisfies the self-consistency relation \eqref{eq:general-fixed-point}. Moreover, \ref{hyp:H1} implies that \(\alpha\in C_b^2(\RR^d;\RR^d)\). The selected invariant law need not be globally attractive.

An apparent discrepancy may arise between the long-time behavior of the finite-particle system and that of its McKean--Vlasov limit. For every fixed \(N\), the particle system admits a unique invariant probability measure \(\pi^N\), whereas the nonlinear limiting equation may admit several invariant laws. When several of these invariant laws are stable, they provide natural candidates for the metastable states of the finite system. Figure~\ref{fig:mospl} provides a numerical illustration of this behavior for the one-dimensional Ornstein--Uhlenbeck--cosine model studied in Section~\ref{sec:ou}. 

For this simulation, we initialize the particles independently with law $\nu_\infty^\alpha= \NNN(\alpha,1/2)$, where $\alpha\simeq 6.411174$ is a stable equilibrium mean. Figure~\ref{fig:mospl} shows the Wasserstein distance between the empirical measure and this fixed equilibrium.
\begin{figure}[htbp]
\centering
\includegraphics[width=0.8\textwidth]{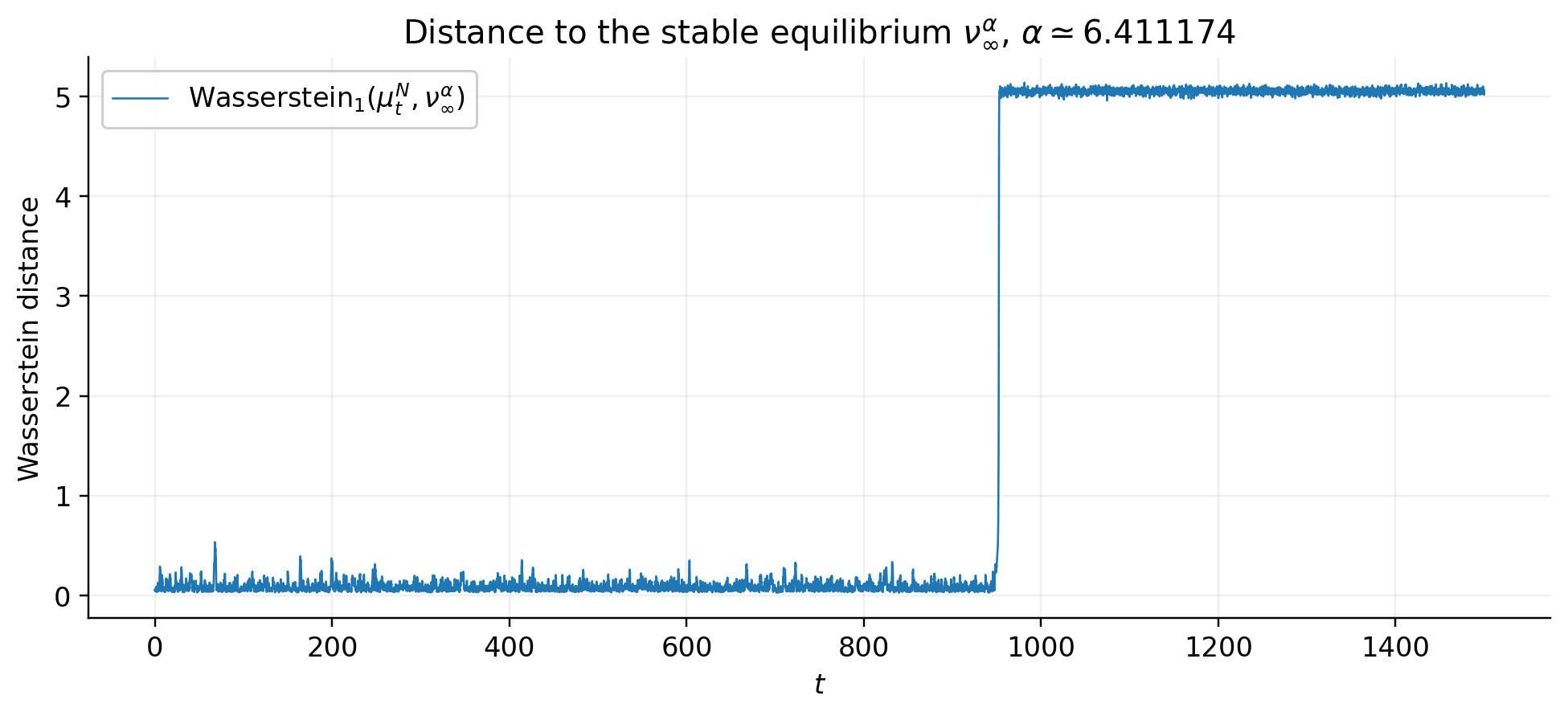}
\caption{Ornstein--Uhlenbeck--cosine model with $J=8.3$, $N=160$ and $\sigma=1$. Wasserstein distance of order~$1$ between the empirical measure $\mu_t^N$ and the stable equilibrium $\nu_\infty^\alpha= \NNN(\alpha,1/2)$, with $\alpha\simeq6.411174$.}
\label{fig:mospl}
\end{figure}

\paragraph{Stability}
We now discuss the stability of the selected invariant law \(\nu_{\infty}^{\alpha}\). We say that \(\nu_{\infty}^{\alpha}\) is locally exponentially stable in \(\Wasserstein_1\) if there exist constants \(C\geq1\), \(\lambda>0\), and \(\varepsilon>0\) such that, for every initial law \(\nu_0\in\PPP_1(\RR^d)\) satisfying \(\Wasserstein_1(\nu_0,\nu_{\infty}^{\alpha})<\varepsilon\), the solution of the nonlinear McKean--Vlasov equation satisfies
\begin{equation}\label{eq:non-linear-stability}
\Wasserstein_1\parent{\Law(X_t^{\nu_0}),\nu_{\infty}^{\alpha}}\leq Ce^{-\lambda t}\Wasserstein_1(\nu_0,\nu_{\infty}^{\alpha}),\qquad t\geq0.
\end{equation}
This estimate means that solutions starting sufficiently close to \(\nu_{\infty}^{\alpha}\) converge to it exponentially fast. The proof of Theorem~\ref{thm:main} below does not use the nonlinear estimate \eqref{eq:non-linear-stability} directly. Instead, it relies on the following linear stability assumption, denoted by \ref{hyp:H3}, for the dynamics linearized at \(\nu_{\infty}^{\alpha}\). Under the present assumptions, this condition is equivalent to \eqref{eq:non-linear-stability}; see Proposition~\ref{prop:stability-equivalence}.

For the analysis of the linearized dynamics, we use the class
\begin{equation}\label{eq:X-class}
\XXX:=\left\{g\in C^2(\RR^d;\RR):\norme{\nabla g}{\infty}+\norme{\nabla^2g}{\infty}<\infty\right\}.
\end{equation}
These functions need not be bounded. Restricting the supremum in \eqref{eq:W1} to $\XXX$ leaves the Wasserstein distance unchanged, since convolution approximation preserves the Lipschitz constant. This allows us to work with twice continuously differentiable test functions throughout the construction, without changing the distance used in the theorem.

Let $(\S_t)_{t\geq0}$ be the linearized semigroup constructed in Section~\ref{sec:lin}. We assume that
\begin{enumerate}[resume*=mes_hyps]
\item \label{hyp:H3} There exist constants $C_S\geq1$ and $\lambda_S>0$ such that, for every $g\in\XXX$,
\begin{equation}\label{eq:linear-stability}
\norme{\nabla \S_tg}{\infty}\leq C_Se^{-\lambda_St}\norme{\nabla g}{\infty},\qquad t\geq0.
\end{equation}
\end{enumerate}

\subsection{Main results}

To state our quantitative confinement result, we first define the dual norm used to measure the deviation of the empirical measure from \(\nu_{\infty}^{\alpha}\).

For every finite signed measure $\eta$ with zero total mass and finite first moment ($\int |x|\,|\eta|(\d x)<\infty$), define
\begin{equation}\label{eq:dualnorm}
\norme{\eta}{\ast}:=\sup_{\substack{g\in C^1(\RR^d),\ g(0)=0\\\norme{\nabla g}{\infty}\leq1}}\vabs{\ps{g,\eta}} \qquad \text{with} \quad \ps{g,\eta} :=\int g\, \d\eta.
\end{equation}
Thus, $\norme{\mu-\nu}{\ast}=\Wasserstein_1(\mu,\nu)$ for $\mu,\nu\in\PPP_1(\RR^d)$. We set
\begin{equation}\label{eq:D-process}
\nu_t^N:=\mu_t^N-\nu_\infty^\alpha,\qquad D_t^N:=\norme{\nu_t^N}{\ast}.
\end{equation}
Since the particle trajectories are continuous, $(D_t^N)_{t\geq0}$ is a non-negative continuous process.

\begin{thm}[Exponentially long confinement]\label{thm:main}
Assume \ref{hyp:H1}--\ref{hyp:H3}, and fix $M_{\rm init}>0$. There exists $a_0>0$ such that, for every $a\in(0,a_0]$, there are constants $C_a,c_a>0$, independent of $N$ and $T$, such that, for every $N\geq1$ and $T\geq0$,
\begin{equation}\label{eq:main}
\PP\parent{\sup_{0\leq t\leq T}\norme{\mu_t^N-\nu_\infty^\alpha}{\ast}\geq a}\leq C_a(1+T)e^{-c_aN},
\end{equation}
provided
\begin{equation}\label{eq:initial-conditions}
V_0^N:=\frac1N\sum_{i=1}^N|X_0^{i,N}|^2\leq M_{\rm init},\qquad \norme{\mu_0^N-\nu_\infty^\alpha}{\ast}\leq\frac{a}{8C_S}\qquad\text{almost surely}.
\end{equation}
\end{thm}

\begin{rmq}[Random initial data]\label{rmq:random-initial-data}
The constants $C_a$ and $c_a$ in Theorem~\ref{thm:main} depend on the initial moment bound $M_{\rm init}$. However, the almost-sure assumption \eqref{eq:initial-conditions} may be relaxed at the cost of two additional error probabilities. More precisely, for every deterministic $M>0$ and every sufficiently small $a>0$, conditioning on the initial configuration gives constants $C_{a,M},c_{a,M}>0$, independent of $N$ and $T$, such that
\[\PP\parent{\sup_{0\leq t\leq T}\norme{\mu_t^N-\nu_\infty^\alpha}{\ast}\geq a}\leq C_{a,M}(1+T)e^{-c_{a,M}N}+\PP\parent{\norme{\mu_0^N-\nu_\infty^\alpha}{\ast}>\frac{a}{8C_S}}+\PP(V_0^N>M).\]
An exponential estimate is therefore recovered whenever the two initial error probabilities decay exponentially in $N$, for a fixed choice of $M$. This occurs, for example, for independent initial data satisfying suitable concentration and moment assumptions. We do not pursue this question here and state Theorem~\ref{thm:main} under almost-sure initial conditions.
\end{rmq}

\begin{cor}[Exponential lower bound on the mean exit time]
\label{cor:mean-exit-time-lower-bound}
Under the assumptions of Theorem~\ref{thm:main}, fix $a\in(0,a_0]$ and set the first exit time from the ball of radius \(a\) centered at the stable equilibrium $\tau_a^N:=\inf\event{t\geq0:\norme{\mu_t^N-\nu_\infty^\alpha}{\ast}\geq a}$. Then, for all sufficiently large $N$,
\[\EE(\tau_a^N)\geq \frac{e^{c_aN}}{8C_a},\]
where $C_a,c_a$ are the constants of Theorem~\ref{thm:main}.
\end{cor}

This result can be made explicit for the Ornstein--Uhlenbeck--cosine model. In this case, both the invariant laws and their linearized stability are characterized by scalar equations.

\begin{cor}[Exponentially long confinement for the OU-cosine model]\label{cor:OU-confinement}
Let $J\geq0$ and $\alpha\in\RR$ satisfy $\alpha=Je^{-1/4}\cos\alpha$ and $1+\alpha\tan\alpha>0$. Consider the particle system
\begin{equation}\label{eq:OU-particles-intro}
\d X_t^{i,N}=\parent{-X_t^{i,N}+\frac{J}{N}\sum_{j=1}^N\cos(X_t^{j,N})}\d t+\d B_t^i.
\end{equation}
Then, for every $M_{\rm init}<\infty$, there exists $a_0>0$ such that, for every $a\in(0,a_0]$, there are constants $C_{a,\alpha},c_{a,\alpha}>0$ such that, for every $N\geq1$ and $T\geq0$,
\[\PP\parent{\sup_{0\leq t\leq T}\Wasserstein_1\parent{\mu_t^N,\NNN(\alpha,1/2)}\geq a}\leq C_{a,\alpha}(1+T)e^{-c_{a,\alpha}N},\]
provided the initial data satisfy the initial conditions \eqref{eq:initial-conditions} with $\nu_{\infty}^{\alpha}=\NNN(\alpha,1/2)$.
\end{cor}

\paragraph{Contributions}
Our result applies to a broad class of McKean--Vlasov diffusions on $\RR^d$ with smooth bounded interactions, under a linearized stability assumption. The approach requires neither a free-energy gradient-flow structure nor reversibility. The Ornstein--Uhlenbeck--cosine model provides a simple non-gradient illustration of this framework. We also extend the result to the double-well model with linear interaction studied by Tugaut \cite{Tug14}. In the regime where its McKean--Vlasov equation admits three invariant measures, we obtain exponentially long confinement near each of the two stable equilibria.

The proof of the preceding results relies on three main ingredients.

First, we give an explicit construction of the linearized semigroup from the Volterra operators introduced by Cormier \cite{Cor25}, bypassing abstract generation theorems and the characterization of the full generator domain. We construct its action on general test functions, identify it with the first variation of nonlinear observables, and establish its semigroup and regularity properties. This provides the linearized evolution needed for the mild representation of the empirical fluctuations.

Second, we handle the non-compact state space through a separate control of the empirical second moment. Combining this control with approximation on a compact ball allows us to reduce the Wasserstein supremum to a finite family of smooth test functions. For a fixed approximation accuracy and moment bound, this family and its regularity bounds are independent of the number of particles and the time horizon. This independence allows a finite union bound without losing the exponential order of the probability estimates.

Third, we establish an exponential maximal estimate for stochastic convolutions with random predictable kernels depending on the terminal time. Its prefactor grows only linearly with the time horizon, making the estimate suitable for exponentially long intervals. We use it both to control the empirical second moment and to estimate the stochastic term in the mild equation. Together with linear stability and bounds on the quadratic remainder, these ingredients yield the confinement estimate.

\paragraph{Organization of the paper}
Section~\ref{sec:prelim} collects moment and regularity estimates for $(\P_t^\alpha)_{t\geq0}$, together with a Wasserstein estimate and a sensitivity formula for the nonlinear flow. Section~\ref{sec:lin} constructs the linearized semigroup, establishes its regularity properties, and relates its exponential stability to local nonlinear stability. Section~\ref{sec:proof} proves the main confinement theorem. Section~\ref{sec:ou} applies the result to the Ornstein--Uhlenbeck--cosine model, and Section~\ref{sec:double-well} establishes the extension to a double-well model with linear interaction. Finally, Appendix~\ref{app:convolution} proves the exponential maximal estimate for stochastic convolutions, and Appendix~\ref{app:finite-net} constructs the finite smooth net used in the Wasserstein reduction.

\section{Preliminary estimates}\label{sec:prelim}

We first collect moment and regularity estimates for \eqref{eq:mvll}, together with a Wasserstein estimate and a sensitivity formula for the nonlinear flow \eqref{eq:mvl}.

\subsection{Estimates for the linear diffusion}

We denote by $\LLL_\alpha$ and $(\P_t^\alpha)_{t\geq0}$ the generator and Markov semigroup of \eqref{eq:mvll}, respectively. The generator is given by
\begin{equation}\label{eq:def-gen-sg}
\LLL_\alpha g:=\frac12\Delta g+(b+\alpha)\cdot\nabla g.
\end{equation}

\paragraph{Finite second moment of the selected equilibrium $\nu_\infty^\alpha$}
We first show that
\begin{equation}\label{eq:Minfty}
M_\infty:=\int_{\RR^d}|x|^2\,\nu_\infty^\alpha(\d x)<\infty.
\end{equation}
To do so, let $Y^\alpha$ solve \eqref{eq:mvll} with $Y_0^\alpha=x\in\RR^d$, and set $V(x):=|x|^2$. 
The dissipativity of $b$ at infinity, the boundedness of $\alpha$, and Young's inequality give constants $\kappa,C>0$ such that
\begin{equation}\label{eq:frozen-Lyapunov}
\LLL_\alpha V(x) = d+2x\cdot(b(x)+\alpha(x)) \leq C-\kappa V(x).
\end{equation}
Applying Itô's formula and taking expectations yields
\[\frac{\d}{\d t}\EE|Y_t^\alpha|^2\leq-\kappa\EE|Y_t^\alpha|^2+C.\]
Grönwall's lemma therefore gives
\[\EE|Y_t^\alpha|^2\leq e^{-\kappa t}|x|^2+\frac{C}{\kappa}\parent{1-e^{-\kappa t}},\qquad t\geq0.\]
We now use this estimate to control the second moment of $\nu_\infty^\alpha$. Since this moment is not yet known to be finite, we cannot directly integrate the preceding inequality against $\nu_\infty^\alpha$. Instead, we apply \cite[Proposition~4.24]{Hai06}, which justifies this passage by combining invariance with a truncation argument. It follows that $M_\infty\leq C/\kappa<\infty$, which proves \eqref{eq:Minfty}.

\paragraph{Short-time estimates for the Markov semigroup $(\P_t^\alpha)_{t\geq0}$}
Since $\alpha$ is bounded, assumption~\ref{hyp:H1} implies that $b+\alpha$ is dissipative at sufficiently large distances. Therefore, \cite[Corollary~2]{Ebe16} gives constants $C_P\geq1$ and $\lambda_P>0$ such that
\begin{equation}\label{eq:finite-time Wasserstein}
\Wasserstein_1\parent{\Law(Y_t^{\alpha,x}),\Law(Y_t^{\alpha,y})}\leq C_Pe^{-\lambda_Pt}|x-y|,\qquad t\geq0,\quad x,y\in\RR^d,
\end{equation}
where $Y^{\alpha,x}$ denotes the solution of \eqref{eq:mvll} starting from $x$. By the Kantorovich formula, it follows that, for every $g\in C^1(\RR^d)$ with bounded gradient,
\begin{equation}\label{eq:P-gradient}
\norme{\nabla\P_t^\alpha g}{\infty}\leq C_Pe^{-\lambda_Pt}\norme{\nabla g}{\infty},\qquad t\geq0.
\end{equation}

Set $L:=\norme{\nabla(b+\alpha)}{\infty}<\infty$, which is finite since $b$ is globally Lipschitz and $\norme{D\alpha}{\infty}\leq\norme{D_xF}{\infty}$. By the Bismut--Elworthy--Li derivative formula, see for instance \cite[Theorem 2.1 and Remark 2]{EL94}, there exists a constant $C_{L}>0$, depending only on $L$, such that, for every bounded measurable function $g$ and every $0<t\leq1$,
\begin{equation}\label{eq:BEL-bound}
\norme{\nabla\P_t^\alpha g}{\infty}\leq\frac{C_{L}}{\sqrt t}\norme{g}{\infty}.
\end{equation}

We shall also use the following second-order estimate: there exists a constant $C>0$ such that, for every $g\in\XXX$ and every $0<t\leq1$,
\begin{equation}\label{eq:P-hessian}
\norme{\nabla^2\P_t^\alpha g}{\infty}\leq\frac{C}{\sqrt t}\norme{\nabla g}{\infty}.
\end{equation}

\begin{proof}[Proof of \eqref{eq:P-hessian}]
Fix $g\in\XXX$. Setting $u_t:=\nabla\P_t^\alpha g$ and differentiating the Kolmogorov equation in space, we obtain
\[\partial_tu_t=\LLL_\alpha u_t+\nabla(b+\alpha)^\top u_t,\qquad u_0=\nabla g.\]
For fixed $t>0$, consider the map $s\mapsto\P_{t-s}^\alpha u_s$. Differentiating this map gives
\[\frac{\d}{\d s}\parent{\P_{t-s}^\alpha u_s}=\P_{t-s}^\alpha\parent{\partial_su_s-\LLL_\alpha u_s}=\P_{t-s}^\alpha\parent{\nabla(b+\alpha)^\top u_s}.\]
Integrating this identity over $[0,t]$ yields
\begin{equation}\label{eq:gradient-representation}
\nabla\P_t^\alpha g=\P_t^\alpha\nabla g+\int_0^t\P_{t-s}^\alpha\parent{\nabla(b+\alpha)^\top\nabla\P_s^\alpha g}\,\d s,
\end{equation}
where $\P_t^\alpha$ acts componentwise on vector-valued functions. The function $\nabla(b+\alpha)^\top\nabla\P_s^\alpha g$ is bounded, with
\[\norme{\nabla(b+\alpha)^\top\nabla\P_s^\alpha g}{\infty}\leq L\norme{\nabla\P_s^\alpha g}{\infty}\leq LC_Pe^{-\lambda_Ps}\norme{\nabla g}{\infty}.\]
Applying \eqref{eq:BEL-bound} to the terms in \eqref{eq:gradient-representation} justifies differentiation under the time integral, since the resulting singularity $(t-s)^{-1/2}$ is integrable, and gives, for $0<t\leq1$,
\[\begin{aligned}
\norme{\nabla^2\P_t^\alpha g}{\infty}
&\leq \frac{C_L}{\sqrt t}\norme{\nabla g}{\infty}
+C_LL\int_0^t\frac{\norme{\nabla\P_s^\alpha g}{\infty}}{\sqrt{t-s}}\,\d s\\
&\leq C_L\parent{\frac1{\sqrt t}+LC_P\int_0^t\frac{e^{-\lambda_Ps}}{\sqrt{t-s}}\,\d s}\norme{\nabla g}{\infty}\\
&\leq C_L\parent{\frac1{\sqrt t}+2LC_P\sqrt t}\norme{\nabla g}{\infty}\\
&\leq \frac{C_L(1+2LC_P)}{\sqrt t}\norme{\nabla g}{\infty}.
\end{aligned}\]
This proves \eqref{eq:P-hessian}.
\end{proof}

\subsection{Estimates for the nonlinear flow}

We denote the nonlinear law flow associated with \eqref{eq:mvl} by
\begin{equation}\label{eq:nonlinear-flow}
\TTT_t(\mu):=\Law(X_t^\mu),\qquad \mu\in\PPP_1(\RR^d),\quad t\geq0.
\end{equation}

\paragraph{Wasserstein estimate for \eqref{eq:mvl}} 
Under assumptions \ref{hyp:H1} and \ref{hyp:H2}, synchronous coupling and Grönwall's lemma give, for every $T>0$, a constant $C_T<\infty$, independent of $\mu,\nu\in\PPP_1(\RR^d)$, such that
\begin{equation}\label{eq:controle-mf}
\Wasserstein_1(\TTT_t(\mu),\TTT_t(\nu))\leq C_T\Wasserstein_1(\mu,\nu),\qquad 0\leq t\leq T.
\end{equation}
Indeed, coupling the initial conditions optimally and driving both solutions with the same Brownian motion, the Lipschitz bounds on the coefficients and Grönwall's lemma yield \eqref{eq:controle-mf}.

\paragraph{The sensitivity formula for \eqref{eq:mvl}}
Let $\mu\in\PPP_1(\RR^d)$ and write $\mu_s:=\TTT_s(\mu)$. For every $g\in\XXX$ and every $t\geq0$,
\begin{equation}\label{eq:sensitivity-formula}
\ps{g,\TTT_t(\mu)}-\ps{g,\TTT_t(\nu_\infty^\alpha)}=\ps{\P_t^\alpha g,\mu-\nu_\infty^\alpha}+\int_0^t\int_{\RR^d}\nabla\P_{t-s}^\alpha g(y)\cdot\parent{\Phi(\mu_s)(y)-\alpha(y)}\,\mu_s(\d y)\,\d s.
\end{equation}
This identity follows from \cite[Proposition~2.11]{Cor25}, applied with initial law $\mu$ and drift perturbation $\Phi(\mu_s)-\alpha$, together with the invariance of $\nu_\infty^\alpha$ under $(\P_t^\alpha)_{t\geq0}$. For completeness, we give the short proof using Itô's formula, as in \cite[Remark~2.15]{Cor25}.

\begin{proof}
The case $t=0$ is immediate. Fix $g\in\XXX$, $t>0$, and $\mu\in\PPP_1(\RR^d)$. Applying Itô's formula to $\P_{t-s}^\alpha g(X_s^\mu)$ and using the backward equation $\partial_s\P_{t-s}^\alpha g=-\LLL_\alpha\P_{t-s}^\alpha g$ gives
\[\d\parent{\P_{t-s}^\alpha g(X_s^\mu)}=\nabla\P_{t-s}^\alpha g(X_s^\mu)\cdot\parent{\Phi(\mu_s)(X_s^\mu)-\alpha(X_s^\mu)}\,\d s+\nabla\P_{t-s}^\alpha g(X_s^\mu)\cdot\d B_s.\]
Integrating between $0$ and $t$ and taking expectations therefore gives
\[\ps{g,\TTT_t(\mu)}=\ps{\P_t^\alpha g,\mu}+\int_0^t\int_{\RR^d}\nabla\P_{t-s}^\alpha g(y)\cdot\parent{\Phi(\mu_s)(y)-\alpha(y)}\,\mu_s(\d y)\,\d s.\]
Finally, the invariance of $\nu_\infty^\alpha$ for both dynamics gives $\ps{g,\TTT_t(\nu_\infty^\alpha)}=\ps{g,\nu_\infty^\alpha}=\ps{\P_t^\alpha g,\nu_\infty^\alpha}$. Subtracting this identity from the preceding one proves \eqref{eq:sensitivity-formula}.

\end{proof}

Writing $\nu_s:=\mu_s-\nu_\infty^\alpha$ and splitting $\mu_s=\nu_\infty^\alpha+\nu_s$ in \eqref{eq:sensitivity-formula} separates the integral into a term linear in $\nu_s$ and a term quadratic in $\nu_s$. Once the linearized semigroup has been constructed, resolving the linear term gives the representation stated in Remark~\ref{lem:nonlinear-variation-of-constants}.

\section{The linearized Fokker--Planck semigroup} \label{sec:lin}

\subsection{Construction of the linearized semigroup}

We now define the central object of our analysis, the semigroup $(\S_t)_{t\geq0}$. We first give an informal presentation of the linearization around $\nu_\infty^\alpha$, before proceeding to the rigorous construction of this semigroup.

\paragraph{Heuristic derivation}
The stability analysis of the selected invariant measure $\nu_\infty^\alpha$ is based on the Fokker--Planck equation linearized around it. Fix an initial law $\mu_0\in\PPP_1(\RR^d)$ and write $\mu_t:=\Law(X_t^{\mu_0})$, where $X^{\mu_0}$ solves \eqref{eq:mvl}. In the sense of distributions, $(\mu_t)_{t\geq0}$ satisfies the nonlinear Fokker--Planck equation
\begin{equation}\label{eq:FP-general}
\partial_t\mu_t=\frac12\Delta\mu_t-\nabla\!\cdot\parent{\parent{b+\Phi(\mu_t)}\mu_t}.
\end{equation}
Set $\nu_t:=\mu_t-\nu_\infty^\alpha$. Since $\Phi$ is linear, $\Phi(\mu_t)=\alpha+\Phi(\nu_t)$. Moreover, $\nu_\infty^\alpha$ satisfies $\LLL_\alpha^\ast\nu_\infty^\alpha=0$. Subtracting this stationary equation from \eqref{eq:FP-general} therefore gives
\begin{equation}\label{eq:linearized-forward}
\partial_t\nu_t=\LLL_\alpha^\ast\nu_t-\nabla\!\cdot\parent{\nu_\infty^\alpha\Phi(\nu_t)}-\nabla\!\cdot\parent{\nu_t\Phi(\nu_t)}.
\end{equation}
The last term is quadratic in $\nu_t$. The linearized operator acting on signed measures is thus
\[\AAA^\ast\eta:=\LLL_\alpha^\ast\eta-\nabla\!\cdot\parent{\nu_\infty^\alpha\Phi(\eta)}.\]
By duality, the corresponding operator acting on test functions $g\in\XXX$ is
\begin{equation}\label{eq:A-dual}
\AAA g:=\LLL_\alpha g+\KKK g,\qquad \KKK g(y):=\int_{\RR^d}F(x,y)\cdot\nabla g(x)\,\nu_\infty^\alpha(\d x).
\end{equation}
In the rank-one case $d=1$ and $F(x,y)=\phi(y)$, this reduces to $\AAA g=\LLL_\alpha g+\phi\ps{g',\nu_\infty^\alpha}$. 
The operator $\KKK$ annihilates constants. Its following properties only require bounded first and second derivatives of $F$ and $\nu_\infty^\alpha\in\PPP_1(\RR^d)$. Indeed, for every $g\in\XXX$,
\begin{equation}\label{eq:K-properties}
\begin{aligned}
|\KKK g(y)|
&\leq\parent{|F(0,0)|+L_F\int_{\RR^d}|x|\,\nu_\infty^\alpha(\d x)+L_F|y|}\norme{\nabla g}{\infty},\\
\norme{\nabla\KKK g}{\infty}+\norme{\nabla^2\KKK g}{\infty}
&\leq L_F\norme{\nabla g}{\infty}.
\end{aligned}
\end{equation}
The first bound follows from the Lipschitz property of $F$, while the second follows by differentiation under the integral. Consequently, $\KKK g\in\XXX$, with at most linear growth. These estimates do not require $F$ to be bounded.

At this stage, the linearization gives a candidate infinitesimal operator $\AAA$. It remains to construct the corresponding evolution and to verify that it is a semigroup whose infinitesimal operator agrees with $\AAA$ on smooth test functions. We already know the evolution associated with $\LLL_\alpha$ for \eqref{eq:mvll}, namely $(\P_t^\alpha)_{t\geq0}$. Observe that differentiating the expected observable with respect to the initial law recovers the action of $\P_t^\alpha$ on $g$. Indeed, setting $u_t^g(\mu):=\EE_\mu[g(Y_t^\alpha)]=\ps{\P_t^\alpha g,\mu}$, we have
\[u_t^g(\mu)-u_t^g(\nu_\infty^\alpha)=\ps{\P_t^\alpha g, \, \mu-\nu_\infty^\alpha}.\]
Since $u_t^g$ is linear in $\mu$, $\P_t^\alpha g$ is a representative of the flat derivative
\[\frac{\delta u_t^g}{\delta m}(\nu_\infty^\alpha)(y)=\P_t^\alpha g(y).\]

We now follow the same idea for \eqref{eq:mvl}, by formally linearizing $v_t^g(\mu):=\ps{g,\TTT_t(\mu)}$ around $\nu_\infty^\alpha$. Write $\mu_s:=\TTT_s(\mu)$. Since $\TTT_t(\nu_\infty^\alpha)=\nu_\infty^\alpha$, the sensitivity formula \eqref{eq:sensitivity-formula} gives
\[v_t^g(\mu)-v_t^g(\nu_\infty^\alpha)=\ps{\P_t^\alpha g, \,\mu-\nu_\infty^\alpha}+\int_0^t\int_{\RR^d}\nabla\P_{t-s}^\alpha g(x)\cdot\Phi(\mu_s-\nu_\infty^\alpha)(x)\,\mu_s(\d x)\,\d s.\]
Writing $F^x(z):=F(x,z)$, differentiating this identity at $\nu_\infty^\alpha$ gives formally, up to an additive constant in $y$,
\[\frac{\delta v_t^g}{\delta m}(\nu_\infty^\alpha)(y)=\P_t^\alpha g(y)+\int_0^t\int_{\RR^d}\nabla\P_{t-s}^\alpha g(x)\cdot\frac{\delta v_s^{F^x}}{\delta m}(\nu_\infty^\alpha)(y)\,\nu_\infty^\alpha(\d x)\,\d s.\]
Indeed, $\Phi(\mu_s)(x)=v_s^{F^x}(\mu)$, which explains the derivative appearing in the integral. The variation of the measure $\mu_s$ contributes nothing at first order, since it is multiplied by $\Phi(\mu_s-\nu_\infty^\alpha)$, which vanishes at equilibrium.
This suggests seeking $\S_tg$ as a representative of the flat derivative of the nonlinear evolution at equilibrium
\[\S_tg(y)=\frac{\delta v_t^g}{\delta m}(\nu_\infty^\alpha)(y).\]
We will show below that $(\S_t)_{t\geq0}$ naturally inherits the semigroup property of the nonlinear law flow $(\TTT_t)_{t\geq0}$.
 In particular, taking the interaction observables $g=F^x$ componentwise gives a closed linear Volterra equation for the family $\S_tF^x$
\[\S_tF^x(y)=\P_t^\alpha F^x(y)+\int_0^t\int_{\RR^d}D_z\parent{\P_{t-s}^\alpha F^x}(z)\S_sF^z(y)\,\nu_\infty^\alpha(\d z)\,\d s.\]
The inner integral is precisely $\Theta_{t-s}\parent{z\mapsto\S_sF^z(y)}(x)$, with the operators $\Theta_t$ introduced in Step~1. Their Volterra resolvent $\Omega_t$, introduced in Step~2, will allow us to solve this equation. Once $\S_tF^x$ is constructed, the preceding formula for a general observable determines $\S_tg$.

We now make this construction rigorous. The identification with the flat derivative and the semigroup property will be justified afterwards.

\paragraph{Construction} 
In classical abstract semigroup theory, one chooses a Banach space and a domain for the operator, and then applies a semigroup generation theorem. Here, we provide a new explicit construction of the semigroup associated with $\AAA$, bypassing abstract generation theorems and the characterization of the full generator domain. The construction proceeds directly through Volterra operators, which describe how the linearized interaction modifies the evolution under $(\P_t^\alpha)_{t\geq0}$. Its explicit nature makes the role of the semigroup clear. Our construction of $\S_t$ is given in Step~3, using the operators introduced in Steps~1 and~2 following \cite{Cor25}. In Step~4, we verify that the resulting family is indeed a semigroup and that its infinitesimal operator agrees with $\AAA$ on the test functions considered.

Set $\H:=L^2(\nu_\infty^\alpha, \RR^d)$, equipped with the norm $\norme{h}{2}:=\parent{\int_{\RR^d}|h(y)|^2\,\nu_\infty^\alpha(\d y)}^{1/2}$. We denote by $\H^\ast$ its continuous dual, and by $\mathcal L(\H)$ the space of bounded linear operators from $\H$ to itself.

\paragraph{Step 1. The operators associated with \eqref{eq:mvll}}
For every $x\in\RR^d$, write $F^x(y):=F(x,y)$. For a particle starting from $y$, its expected influence at time $t$ on a particle located at $x$ is given by $\P_t^\alpha F^x(y)$. If we perturb its initial position from $y$ to $y+\varepsilon h(y)$, the first-order change in this influence is $\varepsilon \nabla_y\parent{\P_t^\alpha F^x}(y)h(y)$. Averaging this contribution over the initial equilibrium law $\nu_\infty^\alpha$ leads to the operator
\begin{equation}\label{eq:Theta-definition}
\Theta_t(h)(x):=\int_{\RR^d}\nabla_y\parent{\P_t^\alpha F^x}(y)h(y)\,\nu_\infty^\alpha(\d y),\qquad h\in\H,\quad t\geq0.
\end{equation}
The same construction applied to a test function $g\in\XXX$ gives the scalar-valued linear functional
\begin{equation}\label{eq:Theta-g-definition}
\Theta_t^g(h):=\int_{\RR^d}\nabla\P_t^\alpha g(y)\cdot h(y)\,\nu_\infty^\alpha(\d y),\qquad h\in\H,\quad t\geq0.
\end{equation}
In particular, componentwise,
\[\Theta_t^{F^x}(h)=\Theta_t(h)(x).\]

Note that $\Theta_t$ is an operator on $\H$, whereas $\Theta_t^g$ is a scalar-valued linear functional on $\H$. Indeed, \eqref{eq:P-gradient} and the Cauchy--Schwarz inequality give
\[\norme{\Theta_t(h)}{\infty}\leq e^{Lt}\norme{\nabla_yF}{\infty}\norme{h}{2},\qquad \vabs{\Theta_t^g(h)} \leq e^{Lt}\norme{\nabla g}{\infty}\norme{h}{2}.\]
Since $\nu_\infty^\alpha$ is a probability measure, $\norme{\Theta_t(h)}{2}\leq\norme{\Theta_t(h)}{\infty}$. Thus, $\Theta_t\in\mathcal L(\H)$ and $\Theta_t^g\in\H^\ast$, with operator norms bounded on finite time intervals. 
These bounds will be used to construct the Volterra resolvent below.

We now relate these operators to variations of the initial law for \eqref{eq:mvll}. For $g\in\XXX$, define
\[u_t^g(\mu):=\ps{\P_t^\alpha g,\mu}=\EE_\mu[g(Y_t^\alpha)],\qquad \mu\in\PPP_1(\RR^d),\]
where $Y^\alpha$ solves \eqref{eq:mvll} with initial law $\mu$. 
Since $\P_t^\alpha g$ is continuously differentiable with bounded gradient, the formula for the Lions derivative of a linear functional gives
\[\partial_\mu u_t^g(\mu)(y)=\nabla\P_t^\alpha g(y),\qquad \mu\in\PPP_2(\RR^d),\quad t\geq0,\]
see \cite[Chapter~5]{CD18}. Consequently, for every $h\in\H$ and $t \geq 0$,
\[\Theta_t^g(h)=\int_{\RR^d}\partial_\mu u_t^g(\nu_\infty^\alpha)(y)\cdot h(y)\,\nu_\infty^\alpha(\d y).\]

\paragraph{Step 2. The Volterra resolvent}
We now perturb the initial positions in the same way, but let the particles
evolve according to the McKean--Vlasov equation~\eqref{eq:mvl}. Unlike in Step~1, the interaction field is no longer kept fixed at $\alpha$, but evolves through $\Phi(\mu_t)$, where $\mu_t$ denotes the law at time $t$. We seek an operator $\Omega_t$ such that $\varepsilon\Omega_t(h)(x)$ describes the first-order change in the interaction field at time $t$ and position $x$, induced by perturbing the initial positions from $y$ to $y+\varepsilon h(y)$.

Let $Z$ have law $\nu_\infty^\alpha$, fix $h\in\H$, and consider the perturbed initial law $\mu_0^\varepsilon:=\Law(Z+\varepsilon h(Z))$. Set $\mu_s^\varepsilon:=\Law(X_s^{\mu_0^\varepsilon})$. Applying Itô's formula to $\P_{t-s}^\alpha F^x(X_s^{\mu_0^\varepsilon})$ and taking expectations gives, by invariance of $\nu_\infty^\alpha$,
\[\begin{aligned}
\Phi(\mu_t^\varepsilon)(x)&= \alpha(x) + \int_{\RR^d}\P_t^\alpha F^x(y)\,(\mu_0^\varepsilon-\nu_\infty^\alpha)(\d y)\\
	&\quad+\int_0^t\int_{\RR^d}\nabla_y\parent{\P_{t-s}^\alpha F^x}(y)\parent{\Phi(\mu_s^\varepsilon)(y)-\alpha(y)}\,\mu_s^\varepsilon(\d y)\,\d s.
\end{aligned}\]
To identify the equation satisfied by the first-order response, write formally $\Phi(\mu_s^\varepsilon)-\alpha=\varepsilon\Omega_s(h)+o(\varepsilon)$. The first integral equals $\varepsilon\Theta_t(h)(x)+o(\varepsilon)$. In the time integral, the factor $\Phi(\mu_s^\varepsilon)-\alpha$ is already of order $\varepsilon$, so replacing $\d \mu_s^\varepsilon$ by $\d \nu_\infty^\alpha$ does not change the first-order contribution. By the definition of $\Theta$, identifying the coefficients of $\varepsilon$ therefore motivates defining $\Omega$ as the Volterra resolvent of $\Theta$, satisfying
\begin{equation}\label{eq:Omega-Volterra}
\Omega_t=\Theta_t+\int_0^t\Theta_{t-s}\circ\Omega_s\,\d s
=\Theta_t+\int_0^t\Omega_s\circ\Theta_{t-s}\,\d s.
\end{equation}
The preceding finite-time bounds allow us to construct $\Omega$ by the Neumann series in \cite[Section~2.2, equations~(11)--(12)]{Cor25}, which also gives the second identity. Its interpretation as the first-order response of the nonlinear flow will be justified below.

Once $\Omega$ is constructed, we define, for every $g\in\XXX$, the associated scalar-valued linear functional
\begin{equation}\label{eq:Omega-g-definition}
\Omega_t^g:=\Theta_t^g+\int_0^t\Theta_{t-s}^g\circ\Omega_s\,\d s.
\end{equation}
It also satisfies
\begin{equation}\label{eq:Omega-g-Volterra}
\Omega_t^g=\Theta_t^g+\int_0^t\Omega_{t-s}^g\circ\Theta_s\,\d s.
\end{equation}

Uniqueness among locally bounded solutions of \eqref{eq:Omega-Volterra} and \eqref{eq:Omega-g-Volterra} follows from Grönwall's lemma. Note that $\Omega^g$ is associated with the resolvent $\Omega$, but is not itself the resolvent of $\Theta^g$, since $\Theta_t^g$ is a linear functional from $\H$ to $\RR$, rather than an operator on $\H$.

For the interaction observables $g = F^x$, these identities give, componentwise,
\[\Omega_t^{F^x}(h)=\Omega_t(h)(x).\]
Moreover, the preceding bounds and \eqref{eq:Omega-g-definition} imply that, for every $T>0$, there exists $C_T<\infty$ such that, for every $g\in\XXX$,
\[\sup_{0\leq t\leq T}\norme{\Omega_t^g}{\H^\ast}\leq C_T\norme{\nabla g}{\infty}.\]

\paragraph{Step 3. Construction of $\S_t$}
We now use these operators to construct $(\S_t)_{t\geq0}$ on $\XXX$. We will verify below that this family is a semigroup whose infinitesimal operator agrees with $\AAA$ on $\XXX$. For $g\in\XXX$ and $t\geq0$, define, componentwise for vector-valued functions,
\begin{equation}\label{eq:S-Omega-representation}
\S_tg(y):=
\begin{cases}
\displaystyle \P_t^\alpha F^x(y)+\int_0^t\Omega_{t-s}(p_s^y)(x)\,\d s,
& \text{if }g=F^x,\\[1em]
\displaystyle \P_t^\alpha g(y)+\int_0^t\Theta_{t-s}^g(q_s^y)\,\d s,
& \text{otherwise},
\end{cases}
\end{equation}
where, for $s\geq0$ and $y\in\RR^d$, we set
\[p_s^y(z):=\P_s^\alpha F^z(y),\qquad q_s^y(z):=\S_sF^z(y),\qquad z\in\RR^d.\]
We first construct $\S_tF^x$ through the first formula, which defines the fields $q_s^y$, and then use the second formula for a general test function. Both formulas are well defined, since the boundedness of $F$ gives $p_s^y\in\H$ with $\norme{p_s^y}{2}\leq\norme{F}{\infty}$, the finite-time bounds on $\Omega$ give bounded fields $q_s^y$ on finite time intervals, and those on $\Theta^g$ ensure the integrability of the second formula. When $F(x,y)$ does not depend on $x$, these formulas simplify, as illustrated by \eqref{eq:S-Omega-representation-OU-cos} for the OU--cosine model.

The definition \eqref{eq:S-Omega-representation} is consistent because, taking $g=F^x$ in the second formula, we recover the first. Indeed, \eqref{eq:Omega-Volterra} and Fubini's theorem give
\begin{equation}\label{eq:construction-qty}
q_t^y=p_t^y+\int_0^t\Theta_{t-s}(q_s^y)\,\d s=p_t^y+\int_0^t\Omega_{t-s}(p_s^y)\,\d s.
\end{equation}
Moreover, the spatial regularity of $\P^\alpha$, the gradient and Hessian bounds \eqref{eq:P-gradient}--\eqref{eq:P-hessian}, and the finite-time bounds on the Volterra operators justify differentiating \eqref{eq:S-Omega-representation} twice with respect to $y$ under the integrals. Consequently, for every $g\in\XXX$ and $t\geq0$, $\S_tg$ is twice continuously differentiable with bounded gradient and Hessian, hence $\S_tg\in\XXX$.

We first identify the functionals $\Omega_t^g$ in terms of $\nabla \S_tg$. Then, we establish their relation with the first variation of the nonlinear flow defined by
\[v_t^g(\mu):=\EE[g(X_t^\mu)],\qquad \mu\in\PPP_1(\RR^d).\]
For the interaction observables $g=F^x$, this probabilistic interpretation is established in \cite[Proposition~2.22]{Cor25}.

\begin{lem}[First-order approximation of the interaction field]\label{lem:interaction-expansion}
For every $T>0$, there exists $C_T<\infty$ such that the following estimate holds for every $\mu\in\PPP_1(\RR^d)$. Write $\mu_s:=\TTT_s(\mu)$ and set $\varepsilon:=\Wasserstein_1(\mu,\nu_\infty^\alpha)$. The perturbation of the interaction field $\Phi(\mu_t)-\alpha$ is approximated by $\int_{\RR^d}q_t^y\,(\mu-\nu_\infty^\alpha)(\d y)$ with a quadratic error:
\begin{equation}\label{eq:interaction-quadratic-expansion}
\sup_{0\leq t\leq T}\norme{\Phi(\mu_t)-\alpha-\int_{\RR^d}q_t^y\,(\mu-\nu_\infty^\alpha)(\d y)}{\infty}\leq C_T\varepsilon^2.
\end{equation}
\end{lem}

\begin{proof}
Fix $T>0$ and $\mu\in\PPP_1(\RR^d)$, and use the notation of the statement. By \eqref{eq:controle-mf} and invariance,
\begin{equation} \label{eq:controle-mf-inv}
\Wasserstein_1(\mu_s,\nu_\infty^\alpha)\leq C_T\varepsilon,\qquad 0\leq s\leq T.
\end{equation}
Applying \eqref{eq:W1} to scalar projections of $F(x,\cdot)$ and $\nabla_xF(x,\cdot)$ therefore gives
\begin{equation}\label{eq:interaction-first-order-bounds}
\norme{\Phi(\mu_s)-\alpha}{\infty}\leq\norme{\nabla_yF}{\infty}C_T\varepsilon,\qquad
\norme{\nabla_x\parent{\Phi(\mu_s)-\alpha}}{\infty}\leq\norme{\nabla_x\nabla_yF}{\infty}C_T\varepsilon.
\end{equation}
These bounds, the product rule, and \eqref{eq:W1}, together with \eqref{eq:P-gradient}--\eqref{eq:P-hessian}, give, for $0\leq t\leq T$,
\[\begin{aligned}
& \sup_{x\in\RR^d}\vabs{\int_0^t\int_{\RR^d}D_z\parent{\P_{t-s}^\alpha F^x}(z)\parent{\Phi(\mu_s)(z)-\alpha(z)}\,(\mu_s-\nu_\infty^\alpha)(\d z)\,\d s}\\
	&\qquad \leq C_T\varepsilon^2\int_0^t\parent{1+(t-s)^{-1/2}}\,\d s \leq C_T\varepsilon^2.
\end{aligned}\]
Here, the estimates are uniform in $x$ because the derivatives of $F^x$ are uniformly bounded; the short-time bounds are extended to $[0,T]$ by the semigroup property of $\P^\alpha$.
We now apply \eqref{eq:sensitivity-formula} componentwise to $g=F^x$ and split $\mu_s=\nu_\infty^\alpha+(\mu_s-\nu_\infty^\alpha)$. Since
\[v_t^{F^x}(\mu)-v_t^{F^x}(\nu_\infty^\alpha)=\Phi(\mu_t)(x)-\alpha(x),\]
the quadratic estimate just proved gives
\[\Phi(\mu_t)-\alpha=\int_{\RR^d}p_t^y\,(\mu-\nu_\infty^\alpha)(\d y)+\int_0^t\Theta_{t-s}\parent{\Phi(\mu_s)-\alpha}\,\d s+O_T(\varepsilon^2),\]
uniformly in space and over $t\in[0,T]$. We compare this equation with the identity obtained by integrating \eqref{eq:construction-qty} against $(\mu-\nu_\infty^\alpha)(\d y)$ and applying Fubini's theorem:
\[\int_{\RR^d}q_t^y\,(\mu-\nu_\infty^\alpha)(\d y)=\int_{\RR^d}p_t^y\,(\mu-\nu_\infty^\alpha)(\d y)+\int_0^t\Theta_{t-s}\parent{\int_{\RR^d}q_s^y\,(\mu-\nu_\infty^\alpha)(\d y)}\,\d s.\]
Subtracting these two equations cancels the terms involving $p_t^y$. Using the finite-time bound $\norme{\Theta_r(h)}{\infty}\leq C_T\norme{h}{\infty}$ for bounded vector fields $h$, Grönwall's lemma proves \eqref{eq:interaction-quadratic-expansion}.
\end{proof}

\begin{lem}[Identification with the Volterra functionals]\label{lem:Omega-S-identification}
For every $g\in\XXX$, every $t\geq0$, and every $h\in\H$,
\begin{equation}\label{eq:Omega-S-identification}
\Omega_t^g(h)=\int_{\RR^d}\nabla \S_tg(y)\cdot h(y)\,\nu_\infty^\alpha(\d y).
\end{equation}
Moreover, $\S_tg$ represents the flat derivative of $v_t^g$ at $\nu_\infty^\alpha$, which is defined up to an additive constant. With this choice of representative,
\begin{equation}\label{eq:S-flat-derivative}
\S_tg(y)=\frac{\delta v_t^g}{\delta m}(\nu_\infty^\alpha)(y).
\end{equation}
The restriction of $v_t^g$ to $\PPP_2(\RR^d)$ is Lions differentiable at $\nu_\infty^\alpha$, with a version of its derivative given by
\begin{equation}\label{eq:S-Lions-derivative}
\nabla \S_tg(y)=\partial_\mu v_t^g(\nu_\infty^\alpha)(y).
\end{equation}
\end{lem}

\begin{rmq}[Probabilistic interpretations]
Let $Z$ have law $\nu_\infty^\alpha$. For every $h\in\H$, set $\mu_\varepsilon:=\Law(Z+\varepsilon h(Z))$. The preceding identifications of the Lions derivatives give, for every $g\in\XXX$ and every $t\geq0$,
\begin{equation}\label{eq:Volterra-probabilistic-interpretation}
\Theta_t^g(h)=\lim_{\varepsilon\to0}\frac{u_t^g(\mu_\varepsilon)-\ps{g,\nu_\infty^\alpha}}{\varepsilon},\qquad
\Omega_t^g(h)=\lim_{\varepsilon\to0}\frac{v_t^g(\mu_\varepsilon)-\ps{g,\nu_\infty^\alpha}}{\varepsilon}.
\end{equation}
Here, we used the invariance of $\nu_\infty^\alpha$ for both dynamics, which gives $u_t^g(\nu_\infty^\alpha)=v_t^g(\nu_\infty^\alpha)=\ps{g,\nu_\infty^\alpha}$. In particular, taking $g=F^x$ componentwise and recalling that $\ps{F^x,\nu_\infty^\alpha}=\alpha(x)$, we obtain
\begin{equation}\label{eq:Volterra-interaction-interpretation}
\Theta_t(h)(x)=\lim_{\varepsilon\to0}\frac{u_t^{F^x}(\mu_\varepsilon)-\alpha(x)}{\varepsilon},\qquad
\Omega_t(h)(x)=\lim_{\varepsilon\to0}\frac{v_t^{F^x}(\mu_\varepsilon)-\alpha(x)}{\varepsilon}.
\end{equation}
\end{rmq}

\begin{proof}
Fix $t\geq0$ and $h\in\H$.

We first consider the case $g=F^x$, understood componentwise, for which we show that, for every $x\in\RR^d$,
\[\Omega_t(h)(x)=\int_{\RR^d}\nabla_yq_t^y(x)h(y)\,\nu_\infty^\alpha(\d y).\]
To do so, recall that $q_t^y(x)=\S_tF^x(y)$. By the definition \eqref{eq:S-Omega-representation}, we have
\[q_t^y=p_t^y+\int_0^t\Omega_{t-s}(p_s^y)\,\d s.\]
Differentiating this representation with respect to $y$ and using the linearity of $\Omega_{t-s}$ yields
\[\nabla_yq_t^y=\nabla_yp_t^y+\int_0^t\Omega_{t-s}(\nabla_yp_s^y)\,\d s.\]
We now integrate this equation against $h(y)\,\nu_\infty^\alpha(\d y)$ and consider separately the two terms on the right-hand side. For the first term, the definitions of $p_t^y$ and $\Theta_t$ give, for every $x\in\RR^d$,
\[\int_{\RR^d}\nabla_yp_t^y(x)h(y)\,\nu_\infty^\alpha(\d y)=\int_{\RR^d}\nabla_y\parent{\P_t^\alpha F^x}(y)h(y)\,\nu_\infty^\alpha(\d y)=\Theta_t(h)(x).\]
For the second term, Fubini's theorem and the linearity of $\Omega_{t-s}$ give, as an identity in $\H$,
\[\begin{aligned}\int_{\RR^d}\parent{\int_0^t\Omega_{t-s}(\nabla_yp_s^y)\,\d s}h(y)\,\nu_\infty^\alpha(\d y)&=\int_0^t\Omega_{t-s}\parent{\int_{\RR^d}\nabla_yp_s^y h(y)\,\nu_\infty^\alpha(\d y)}\,\d s\\&=\int_0^t\Omega_{t-s}\parent{\Theta_s(h)}\,\d s.\end{aligned}\]
Combining these identities and using \eqref{eq:Omega-Volterra}, we obtain
\[\int_{\RR^d}\nabla_yq_t^y(x)h(y)\,\nu_\infty^\alpha(\d y)=\Theta_t(h)(x)+\int_0^t\Omega_{t-s}\parent{\Theta_s(h)}(x)\,\d s=\Omega_t(h)(x).\]
Since $q_t^y(x)=\S_tF^x(y)$ and $\Omega_t^{F^x}(h)=\Omega_t(h)(x)$ componentwise, this proves the identity for the case $g = F^x$.

We now consider $g\in\XXX$. Differentiating its defining formula in \eqref{eq:S-Omega-representation}, integrating against $h(y)\,\nu_\infty^\alpha(\d y)$, and using the linearity of $\Theta_{t-s}^g$ yields
\[\int_{\RR^d}\nabla \S_tg(y)\cdot h(y)\,\nu_\infty^\alpha(\d y)=\int_{\RR^d}\nabla\P_t^\alpha g(y)\cdot h(y)\,\nu_\infty^\alpha(\d y)+\int_0^t\Theta_{t-s}^g\parent{\int_{\RR^d}\nabla_yq_s^y h(y)\,\nu_\infty^\alpha(\d y)}\,\d s.\]
Hence, the preceding identity and \eqref{eq:Omega-g-definition} give
\[\int_{\RR^d}\nabla \S_tg(y)\cdot h(y)\,\nu_\infty^\alpha(\d y)=\Theta_t^g(h)+\int_0^t\Theta_{t-s}^g\parent{\Omega_s(h)}\,\d s=\Omega_t^g(h).\]
This proves \eqref{eq:Omega-S-identification}.

We now prove \eqref{eq:S-flat-derivative}. Fix $T>0$ and $\mu\in\PPP_1(\RR^d)$, write $\mu_s:=\TTT_s(\mu)$, and set $\varepsilon:=\Wasserstein_1(\mu,\nu_\infty^\alpha)$.
For $g\in\XXX$ and $0\leq t\leq T$, we rewrite \eqref{eq:sensitivity-formula} to isolate the expected first-order term. Splitting $\mu_s=\nu_\infty^\alpha+(\mu_s-\nu_\infty^\alpha)$ and adding and subtracting the term involving $q_s^y$ gives
\[\begin{aligned}
v_t^g(\mu)-v_t^g(\nu_\infty^\alpha) & = \ps{\P_t^\alpha g,\mu-\nu_\infty^\alpha}+\int_0^t\Theta_{t-s}^g\parent{\int_{\RR^d}q_s^y\,(\mu-\nu_\infty^\alpha)(\d y)}\,\d s\\
	&\qquad+\int_0^t\Theta_{t-s}^g\parent{\Phi(\mu_s)-\alpha-\int_{\RR^d}q_s^y\,(\mu-\nu_\infty^\alpha)(\d y)}\,\d s\\
	&\qquad+\int_0^t\ps{\nabla\P_{t-s}^\alpha g\cdot\parent{\Phi(\mu_s)-\alpha},\mu_s-\nu_\infty^\alpha}\,\d s.
\end{aligned}\]
The first line is precisely the linear term we seek. Indeed, Fubini's theorem and the definition of $\S_tg$ give
\[\begin{aligned}
& \ps{\P_t^\alpha g,\mu-\nu_\infty^\alpha}+\int_0^t\Theta_{t-s}^g\parent{\int_{\RR^d}q_s^y\,(\mu-\nu_\infty^\alpha)(\d y)}\,\d s\\
	&\qquad =\int_{\RR^d}\parent{\P_t^\alpha g(y)+\int_0^t\Theta_{t-s}^g(q_s^y)\,\d s}\,(\mu-\nu_\infty^\alpha)(\d y)\\
	&\qquad =\int_{\RR^d}\S_tg(y)\,(\mu-\nu_\infty^\alpha)(\d y).\end{aligned}\]
It therefore remains to show that each of the two remaining integrals is $O_T\parent{\norme{\nabla g}{\infty}\varepsilon^2}$.

We first estimate the last integral. By \eqref{eq:W1},
\[\vabs{\ps{\nabla\P_{t-s}^\alpha g\cdot\parent{\Phi(\mu_s)-\alpha},\mu_s-\nu_\infty^\alpha}}\leq\Lip\parent{\nabla\P_{t-s}^\alpha g\cdot\parent{\Phi(\mu_s)-\alpha}}\Wasserstein_1(\mu_s,\nu_\infty^\alpha).\]
The product rule gives, for $0\leq s<t\leq T$,
\[\Lip\parent{\nabla\P_{t-s}^\alpha g\cdot\parent{\Phi(\mu_s)-\alpha}}\leq\norme{\nabla^2\P_{t-s}^\alpha g}{\infty}\norme{\Phi(\mu_s)-\alpha}{\infty}+\norme{\nabla\P_{t-s}^\alpha g}{\infty}\norme{\nabla_x\parent{\Phi(\mu_s)-\alpha}}{\infty}.\]
Using \eqref{eq:interaction-first-order-bounds} from Lemma~\ref{lem:interaction-expansion}, together with \eqref{eq:P-gradient}, \eqref{eq:P-hessian}, and the semigroup property of $\P^\alpha$, we obtain
\[\Lip\parent{\nabla\P_{t-s}^\alpha g\cdot\parent{\Phi(\mu_s)-\alpha}}\leq C_T\norme{\nabla g}{\infty}\varepsilon\parent{1+(t-s)^{-1/2}},\qquad 0\leq s<t\leq T.\]
Combining these estimates with \eqref{eq:controle-mf-inv} and integrating over $s$ gives, after increasing $C_T$,
\[\begin{aligned}\vabs{\int_0^t\ps{\nabla\P_{t-s}^\alpha g\cdot\parent{\Phi(\mu_s)-\alpha},\mu_s-\nu_\infty^\alpha}\,\d s}&\leq C_T\norme{\nabla g}{\infty}\varepsilon^2\int_0^t\parent{1+(t-s)^{-1/2}}\,\d s\\&\leq C_T\norme{\nabla g}{\infty}\varepsilon^2,\end{aligned}\]
where we used $\int_0^t(1+(t-s)^{-1/2})\,\d s=t+2\sqrt t\leq T+2\sqrt T$.

To control the other integral, we use the quadratic approximation \eqref{eq:interaction-quadratic-expansion} from Lemma~\ref{lem:interaction-expansion}. Together with the finite-time bounds on $\Theta^g$, it gives
\[\begin{aligned}
& \vabs{\int_0^t\Theta_{t-s}^g\parent{\Phi(\mu_s)-\alpha-\int_{\RR^d}q_s^y\,(\mu-\nu_\infty^\alpha)(\d y)}\,\d s}\\
&\qquad \leq C_T\norme{\nabla g}{\infty}\int_0^t\norme{\Phi(\mu_s)-\alpha-\int_{\RR^d}q_s^y\,(\mu-\nu_\infty^\alpha)(\d y)}{\infty}\,\d s\\
&\qquad \leq C_T\norme{\nabla g}{\infty}\varepsilon^2,
\end{aligned}\]
after increasing $C_T$. Combining the two remainder estimates with the initial decomposition gives
\begin{equation}\label{eq:nonlinear-first-order-expansion}
v_t^g(\mu)-v_t^g(\nu_\infty^\alpha)=\int_{\RR^d}\S_tg(y)\,(\mu-\nu_\infty^\alpha)(\d y)+O_T\parent{\norme{\nabla g}{\infty}\varepsilon^2}.
\end{equation}
Thus, $\S_tg$ represents the flat derivative at $\nu_\infty^\alpha$, up to an additive constant, proving \eqref{eq:S-flat-derivative}.

We finally identify the Lions derivative \eqref{eq:S-Lions-derivative}. Fix $g\in\XXX$ and $t>0$. Let $Z$ and $H$ be $\RR^d$-valued square-integrable random variables on the same atomless probability space, with $\Law(Z)=\nu_\infty^\alpha$, and set $\mu^H:=\Law(Z+H)$. The coupling $(Z+H,Z)$ gives
\[\Wasserstein_1(\mu^H,\nu_\infty^\alpha)\leq\EE|H|\leq\norme{H}{L^2}.\]
Since $\S_tg$ is continuously differentiable with bounded gradient, the map $H\mapsto\EE[\nabla\S_tg(Z)\cdot H]$ is a bounded linear functional on $L^2$, and a first-order Taylor expansion under the expectation, together with \eqref{eq:nonlinear-first-order-expansion}, gives, as $\norme{H}{L^2}\to0$,
\[\begin{aligned}v_t^g(\mu^H)-v_t^g(\nu_\infty^\alpha)&=\EE\parent{\S_tg(Z+H)-\S_tg(Z)}+O_{t,g}\parent{\norme{H}{L^2}^2}\\&=\EE\parent{\nabla\S_tg(Z)\cdot H}+o\parent{\norme{H}{L^2}}.\end{aligned}\]
This proves Lions differentiability, with a version of the derivative given by \eqref{eq:S-Lions-derivative}.
\end{proof}

\paragraph{Step 4. The semigroup property and the infinitesimal operator}
We now show that the constructed family is a semigroup and that its infinitesimal operator agrees with $\AAA$ on $\XXX$. We first use the nonlinear flow to prove the semigroup property. We then identify the infinitesimal operator and derive the Duhamel formula. Unlike $(\P_t^\alpha)_{t\geq0}$, the linearized semigroup $(\S_t)_{t\geq0}$ is not, in general, a Markov semigroup: although it preserves constants, it need not preserve positivity.

\begin{prop}[Linearized Fokker--Planck semigroup]\label{prop:sgfpl}
The family $(\S_t)_{t\geq0}$ constructed above satisfies the following properties.
\begin{enumerate}
	\item[(i)] It is a linear semigroup on $\XXX$:
\[\S_0g=g,\qquad \S_{t+r}g=\S_t(\S_rg),\qquad g\in\XXX,\quad t,r\geq0.\]
	\item[(ii)] For every $g\in\XXX$ and every $y\in\RR^d$,
\begin{equation}\label{eq:S-generator}
\lim_{h\to 0}\frac{\S_hg(y)-g(y)}{h}=\LLL_\alpha g(y) +\KKK g(y) =\AAA g(y).
\end{equation}
Moreover, $\S_t g$ solves the Kolmogorov equation
\begin{equation}\label{eq:S-backwar-equation}
\partial_t\S_tg=\AAA \S_tg,\qquad t\geq0,\qquad \S_0g=g.
\end{equation}
	\item[(iii)] For every $g\in\XXX$ and $t\geq0$,
\begin{equation}\label{eq:S-duhamel}
\S_tg=\P_t^\alpha g+\int_0^t\P_{t-s}^\alpha\KKK \S_sg\,\d s.
\end{equation}
\end{enumerate}
\end{prop}

\begin{proof}

We first prove \textup{(i)}. Fix $t,s>0$, $\mu\in \PPP_1(\RR^d)$ and $g\in\XXX$. Linearity follows directly from the construction, and $\S_0g=g$ by definition. Set $\varepsilon:=\Wasserstein_1(\mu,\nu_\infty^\alpha)$. The flow property gives $\TTT_{t+s}(\mu)=\TTT_t(\TTT_s(\mu))$. Consequently,
\[v_{t+s}^g(\mu)=v_t^g(\TTT_s(\mu)).\]
For the left-hand side, \eqref{eq:S-flat-derivative} gives
\[\begin{aligned}
v_{t+s}^g(\mu)-v_{t+s}^g(\nu_\infty^\alpha)&=\int_{\RR^d}\frac{\delta v_{t+s}^g}{\delta m}(\nu_\infty^\alpha)(x)\,(\mu-\nu_\infty^\alpha)(\d x)+o(\varepsilon)\\
	&=\int_{\RR^d}\S_{t+s}g(x)\,(\mu-\nu_\infty^\alpha)(\d x)+o(\varepsilon).
\end{aligned}\]
For the right-hand side, applying equations \eqref{eq:controle-mf} and \eqref{eq:S-flat-derivative}, we therefore obtain
\[\begin{aligned}v_t^g(\TTT_s(\mu))-v_t^g(\nu_\infty^\alpha)&=\int_{\RR^d}\frac{\delta v_t^g}{\delta m}(\nu_\infty^\alpha)(x)\,(\TTT_s(\mu)-\nu_\infty^\alpha)(\d x)+o(\varepsilon)\\&=\int_{\RR^d}\S_tg(x)\,(\TTT_s(\mu)-\nu_\infty^\alpha)(\d x)+o(\varepsilon)\\&=\EE[\S_tg(X_s^\mu)]-\EE[\S_tg(X_s^{\nu_\infty^\alpha})]+o(\varepsilon)\\&=v_s^{\S_tg}(\mu)-v_s^{\S_tg}(\nu_\infty^\alpha)+o(\varepsilon).\end{aligned}\]
Since $\S_tg\in\XXX$, we may apply the same expansion once more, now to $v_s^{\S_tg}$, which gives
\[\begin{aligned}v_s^{\S_tg}(\mu)-v_s^{\S_tg}(\nu_\infty^\alpha)&=\int_{\RR^d}\frac{\delta v_s^{\S_tg}}{\delta m}(\nu_\infty^\alpha)(x)\,(\mu-\nu_\infty^\alpha)(\d x)+o(\varepsilon)\\&=\int_{\RR^d}\S_s(\S_tg)(x)\,(\mu-\nu_\infty^\alpha)(\d x)+o(\varepsilon).\end{aligned}\]
Thus,
\[\int_{\RR^d}\parent{\S_{t+s}g(x)-\S_s(\S_tg)(x)} (\mu-\nu_\infty^\alpha)(\d x)=o(\varepsilon).\]
Taking $\mu=(1-h)\nu_\infty^\alpha+h\delta_x$ with $0 < h < 1$, for which $\varepsilon=h\Wasserstein_1(\delta_x,\nu_\infty^\alpha)$, dividing by $h$, and letting $h\to 0$ shows that
\[\S_{t+s}g(x)-\S_s(\S_tg)(x)=\ps{\S_{t+s}g-\S_s(\S_tg),\nu_\infty^\alpha}.\]

It remains to show that $\ps{\S_{t+s}g,\nu_\infty^\alpha}=\ps{\S_s(\S_tg),\nu_\infty^\alpha}$. 
The preceding argument shows that $\S_{t+r}g-\S_r(\S_tg)$ is constant in space, in particular $\nabla \S_r(\S_tg)=\nabla \S_{t+r}g$. Therefore, by \eqref{eq:Omega-S-identification}, for every $h\in\H$, $\Omega_r^{\S_tg}=\Omega_{t+r}^g$.
We know that we have the equivalent representation
\[\S_ug(y)=\P_u^\alpha g(y)+\int_0^u\Omega_r^g(p_{u-r}^y)\,\d r.\]
Integrating the preceding representation against $\nu_\infty^\alpha$ and using its invariance under $(\P_r^\alpha)_{r\geq0}$, which gives $\int_{\RR^d}p_r^y\,\nu_\infty^\alpha(\d y)=\alpha$, we obtain
\[\ps{\S_ug,\nu_\infty^\alpha}=\ps{\P_u^\alpha g,\nu_\infty^\alpha}+\int_0^u\Omega_r^g\parent{\int_{\RR^d}p_{u-r}^y\,\nu_\infty^\alpha(\d y)}\,\d r=\ps{g,\nu_\infty^\alpha}+\int_0^u\Omega_r^g(\alpha)\,\d r.\]
Applying the preceding identity first to $\S_tg$ at time $s$, then to $g$ at time $t$, and using $\Omega_r^{\S_tg}=\Omega_{t+r}^g$, we obtain
\[\begin{aligned}
\ps{\S_s(\S_tg),\nu_\infty^\alpha} &=\ps{\S_tg,\nu_\infty^\alpha}+\int_0^s\Omega_r^{\S_tg}(\alpha)\,\d r\\
	&= \parent{\ps{g,\nu_\infty^\alpha}+\int_0^t\Omega_r^g(\alpha)\,\d r}+\int_0^s\Omega_{t+r}^g(\alpha)\,\d r\\
	&=\ps{g,\nu_\infty^\alpha}+\int_0^t\Omega_r^g(\alpha)\,\d r+\int_t^{t+s}\Omega_r^g(\alpha)\,\d r\\
	&=\ps{g,\nu_\infty^\alpha}+\int_0^{t+s}\Omega_r^g(\alpha)\,\d r \;=\; \ps{\S_{t+s}g,\nu_\infty^\alpha}.
\end{aligned}\]
The constant is therefore zero. The cases where either time is zero follow from $\S_0g=g$. This proves \textup{(i)}.

We next prove \textup{(ii)}. Fix $g\in\XXX$ and $y\in\RR^d$. We have
\[\frac{\S_tg(y)-g(y)}{t}=\frac{\P_t^\alpha g(y)-g(y)}{t}+\frac1t\int_0^t\Theta_{t-s}^g(q_s^y)\,\d s.\]
Since $\LLL_\alpha$ is the infinitesimal operator associated with $(\P_t^\alpha)_{t\geq0}$,
\[\frac{\P_t^\alpha g(y)-g(y)}{t}\xra{}{t\to 0}\LLL_\alpha g(y).\]
Moreover, for any functions $f:[0,\infty)\lra\H^\ast$ and $u:[0,\infty)\lra\H$ that are continuous at zero in their respective norms, a standard averaging argument gives
\[\frac1t\int_0^t f(t-s)\parent{u(s)}\,\d s\xra{}{t\to 0}f(0)\parent{u(0)}.\]
Since $\S_0F^x=F^x$ and $\P_0^\alpha g=g$,
\[\Theta_0^g(q_0^y)=\int_{\RR^d}\S_0F^x(y)\cdot\nabla\P_0^\alpha g(x)\,\nu_\infty^\alpha(\d x)=\int_{\RR^d}F(x,y)\cdot\nabla g(x)\,\nu_\infty^\alpha(\d x)=\KKK g(y).\]
Applying the preceding observation to the continuous maps $r\mapsto\Theta_r^g\in\H^\ast$ and $s\mapsto q_s^y\in\H$, we obtain
\[\frac1t\int_0^t\Theta_{t-s}^g(q_s^y)\,\d s\xra{}{t\to 0}\Theta_0^g(q_0^y)=\KKK g(y).\]
Combining the two limits proves \eqref{eq:S-generator}.

To prove \eqref{eq:S-backwar-equation}, recall that $\S_tg\in\XXX$ for every $t\geq0$. Thus, \textup{(i)} and \eqref{eq:S-generator} give the right derivative
\[\partial_t\S_tg(y)=\lim_{h\to0}\frac{\S_{t+h}g(y)-\S_tg(y)}{h}=\lim_{h\to0}\frac{\S_h(\S_tg)(y)-\S_tg(y)}{h}=\AAA\S_tg(y).\]

We finally prove \textup{(iii)}. Fix $g\in\XXX$ and $t>0$. Consider the map $s\mapsto\P_{t-s}^\alpha \S_sg$. Its values at $s=0$ and $s=t$ are $\P_t^\alpha g$ and $\S_tg$, respectively. Hence,
\[\S_tg-\P_t^\alpha g=\int_0^t\partial_s\parent{\P_{t-s}^\alpha \S_sg}\,\d s.\]
By the product rule and \eqref{eq:S-backwar-equation}, for $0<s<t$,
\[\partial_s\parent{\P_{t-s}^\alpha \S_sg}=-\left.\parent{\partial_r\P_r^\alpha}\right|_{r=t-s}(\S_sg)+\P_{t-s}^\alpha\parent{\partial_s\S_sg}=-\P_{t-s}^\alpha\LLL_\alpha \S_sg+\P_{t-s}^\alpha\AAA \S_sg=\P_{t-s}^\alpha\KKK \S_sg.\]
Substituting this identity into the preceding integral proves \eqref{eq:S-duhamel}.
\end{proof}

We now return to the sensitivity formula \eqref{eq:sensitivity-formula}. Splitting $\mu_s=\nu_\infty^\alpha+\nu_s$ and using the linearized semigroup to resolve the linear Volterra term gives the following representation.
\begin{rmq}[Resolvent form of the sensitivity formula]\label{lem:nonlinear-variation-of-constants}
Under assumptions \ref{hyp:H1} and \ref{hyp:H2}, let $\mu_0\in\PPP_1(\RR^d)$ and define $\mu_t:=\Law(X_t^{\mu_0})$ and $\nu_t:=\mu_t-\nu_\infty^\alpha$. Then, for every $g\in\XXX$ and $t\geq0$,
\begin{equation}\label{eq:nonlinear-variation-of-constants}
\ps{g,\nu_t}=\ps{\S_tg,\nu_0}+\int_0^t\iint_{\RR^d\times\RR^d}\nabla\S_{t-s}g(x)\cdot F(x,y)\,\nu_s(\d x)\nu_s(\d y)\,\d s.
\end{equation}
The first term describes the evolution of the initial perturbation under the linearized equation, while the integral accounts for the quadratic term in the nonlinear equation.
\end{rmq}

\subsection{Regularity estimates for the linearized semigroup}

The construction already provides finite-time regularity and preservation of $\XXX$. We now establish terminal-time and long-time estimates needed for the confinement argument.

For the following terminal-time estimate, we consider test functions $g$ for which $\LLL_\alpha g$ is bounded and globally Lipschitz. This is ensured by the compact support of the gradient, since $b+\alpha$ and its first derivative are bounded on compact sets. We therefore introduce the smaller class
\begin{equation}\label{eq:D-class}
\DDD:=\left\{g\in C^\infty(\RR^d;\RR):\nabla g\in C_c^\infty(\RR^d;\RR^d)\right\}\subset\XXX.
\end{equation}
The finite family of test functions used in the confinement proof will be chosen in $\DDD$.

\begin{lem}[Terminal-time regularity]\label{lem:terminal-time}
Assume \ref{hyp:H3}. For every $g\in\DDD$, there exists a finite constant $C_g$ such that, for every $0\leq h\leq1$ and $r\geq0$,
\begin{equation}\label{eq:terminal-small}
\norme{\nabla(\S_hg-g)}{\infty}\leq C_gh,
\end{equation}
and
\begin{equation}\label{eq:terminal-increment}
\norme{\nabla \S_{r+h}g-\nabla \S_rg}{\infty}\leq C_SC_ghe^{-\lambda_Sr}.
\end{equation}
\end{lem}

\begin{proof}
We estimate separately the two terms on the right-hand side of the following equation given by Duhamel's formula \eqref{eq:S-duhamel},
\[\S_hg-g=(\P_h^\alpha g-g)+\int_0^h\P_{h-s}^\alpha\KKK \S_sg\,\d s.\]
For the first one, since $\partial_u\P_u^\alpha g=\P_u^\alpha\LLL_\alpha g$, we
have $\P_h^\alpha g-g=\int_0^h\P_u^\alpha\LLL_\alpha g\,\d u$. The function $\LLL_\alpha g$ is globally Lipschitz because
\[\nabla\LLL_\alpha g=\frac12\nabla\Delta g+\nabla(b+\alpha)^\top\nabla g+\nabla^2g\,(b+\alpha),\]
and every term on the right-hand side is bounded for $g\in\DDD$. 
Using \eqref{eq:P-gradient}, we obtain
\[\norme{\nabla(\P_h^\alpha g-g)}{\infty}\leq C_P\int_0^he^{-\lambda_Pu}\norme{\nabla\LLL_\alpha g}{\infty}\,\d u\leq C_Ph\norme{\nabla\LLL_\alpha g}{\infty}.\]
For the second term, \eqref{eq:P-gradient}, \eqref{eq:K-properties}, and \ref{hyp:H3} give
\[\norme{\int_0^h\nabla\P_{h-s}^\alpha\KKK\S_sg\,\d s}{\infty} \leq C_P \norme{\nabla_y F}{\infty} \int_0^he^{-\lambda_P(h-s)}\norme{\nabla\S_sg}{\infty}\,\d s \leq C_PC_SL_Fh\norme{\nabla g}{\infty}.\]
Combining these estimates proves \eqref{eq:terminal-small}, with $C_g:=C_P\parent{\norme{\nabla\LLL_\alpha g}{\infty}+C_SL_F\norme{\nabla g}{\infty}}$.

Finally, the semigroup property gives $\S_{r+h}g-\S_rg=\S_r(\S_hg-g)$. Since $\S_hg-g\in\XXX$, we may apply \ref{hyp:H3} to this function and then use \eqref{eq:terminal-small}, obtaining
\[\norme{\nabla \S_{r+h}g-\nabla \S_rg}{\infty}\leq C_Se^{-\lambda_Sr}\norme{\nabla(\S_hg-g)}{\infty}\leq C_SC_ghe^{-\lambda_Sr},\]
which proves \eqref{eq:terminal-increment}.
\end{proof}

\begin{lem}[Hessian regularization]\label{lem:hessian}
Assume that \ref{hyp:H3} holds. For every $\lambda_H\in(0,\lambda_S)$, there exists a constant $C_H=C_H(\lambda_H)>0$ such that, for every $g\in\XXX$ and every $t>0$,
\begin{equation}\label{eq:hessian}
\norme{\nabla^2\S_tg}{\infty}\leq C_Ht^{-1/2}e^{-\lambda_Ht}\norme{\nabla g}{\infty}.
\end{equation}
\end{lem}

\begin{proof}
Fix $g\in\XXX$ and $\lambda_H\in(0,\lambda_S)$. We first establish the estimate for $0<t\leq1$. Since $\S_sg\in\XXX$ and $\KKK\S_sg\in\XXX$, we may differentiate \eqref{eq:S-duhamel} twice in space and apply \eqref{eq:P-hessian}, which gives
\[\begin{aligned}\norme{\nabla^2\S_tg}{\infty}&\leq\frac{C}{\sqrt t}\norme{\nabla g}{\infty}+C\int_0^t\frac{\norme{\nabla\KKK \S_sg}{\infty}}{\sqrt{t-s}}\,\d s\\&\leq\frac{C}{\sqrt t}\norme{\nabla g}{\infty}+CC_S\norme{\nabla_yF}{\infty}\norme{\nabla g}{\infty}\int_0^t(t-s)^{-1/2}e^{-\lambda_Ss}\,\d s\\&\leq\frac{C}{\sqrt t}\norme{\nabla g}{\infty}.\end{aligned}\]
Here, the second inequality follows from \eqref{eq:K-properties} and \ref{hyp:H3}, while the last one uses 
\[ \int_0^t(t-s)^{-1/2}e^{-\lambda_Ss}\,\d s\leq2\sqrt t. \]
Since $e^{-\lambda_H} \leq e^{-\lambda_Ht}$ for $0<t\leq1$, the preceding estimate implies \eqref{eq:hessian} on $(0,1]$ after increasing $C_H$.
We now consider $t\geq1$. By the semigroup property, $\S_tg=\S_1(\S_{t-1}g)$. Since $\S_{t-1}g\in\XXX$, applying the preceding short-time estimate at time $1$ to this function and then using \ref{hyp:H3} gives
\[\norme{\nabla^2\S_tg}{\infty}\leq C\norme{\nabla \S_{t-1}g}{\infty}\leq CC_Se^{-\lambda_S(t-1)}\norme{\nabla g}{\infty}.\]
Since $\lambda_H<\lambda_S$ and $\sup_{t\geq1}\sqrt t\,e^{-(\lambda_S-\lambda_H)t}<\infty$, the preceding estimate implies \eqref{eq:hessian} on $[1,\infty)$ after increasing $C_H$ once more.
Combining the short- and long-time estimates completes the proof.
\end{proof}

\begin{prop}[Equivalence between nonlinear and linearized stability]\label{prop:stability-equivalence}
Under assumptions \ref{hyp:H1} and \ref{hyp:H2}, the following assertions are equivalent.
\begin{enumerate}
\item[$(i)$] The invariant measure $\nu_\infty^\alpha$ is locally stable in $\Wasserstein_1$, in the sense of \eqref{eq:non-linear-stability}.
\item[$(ii)$] The linearized semigroup $(\S_t)_{t\geq0}$ satisfies the exponential stability assumption \ref{hyp:H3}.
\end{enumerate}
\end{prop}

\begin{proof}
We first prove \textup{(i)} $\;\Lra\;$ \textup{(ii)}. Let $C\geq1$, $\lambda>0$, and $\varepsilon_0>0$ be the constants in \eqref{eq:non-linear-stability}. Fix $g\in\XXX$, $t>0$, and $h\in\H$. Let $Z$ have law $\nu_\infty^\alpha$ and set $\mu_\varepsilon:=\Law(Z+\varepsilon h(Z))$. The coupling given by these two random variables yields
\[\Wasserstein_1(\mu_\varepsilon,\nu_\infty^\alpha)\leq\varepsilon\nu_\infty^\alpha(|h|),\qquad\varepsilon>0.\]
For $\varepsilon$ sufficiently small, $\varepsilon\nu_\infty^\alpha(|h|)<\varepsilon_0$. By the definition of $v_t^g$ and the invariance of $\nu_\infty^\alpha$, the Kantorovich formula gives
\[\vabs{v_t^g(\mu_\varepsilon)-v_t^g(\nu_\infty^\alpha)}=\vabs{\ps{g,\Law(X_t^{\mu_\varepsilon})-\nu_\infty^\alpha}}\leq\norme{\nabla g}{\infty}\Wasserstein_1\parent{\Law(X_t^{\mu_\varepsilon}),\nu_\infty^\alpha}.\]
Applying \eqref{eq:non-linear-stability} and using the preceding coupling bound therefore yields
\[\vabs{\frac{v_t^g(\mu_\varepsilon)-v_t^g(\nu_\infty^\alpha)}{\varepsilon}}\leq Ce^{-\lambda t}\norme{\nabla g}{\infty}\frac{\Wasserstein_1(\mu_\varepsilon,\nu_\infty^\alpha)}{\varepsilon}\leq Ce^{-\lambda t}\norme{\nabla g}{\infty}\nu_\infty^\alpha(|h|).\]
By \eqref{eq:S-Lions-derivative} and \eqref{eq:Omega-S-identification}, the first variation of $v_t^g$ is represented by $\Omega_t^g$:
\[\vabs{\Omega_t^g(h)} = %
\lim_{\varepsilon\to 0} \vabs{\frac{v_t^g(\mu_\varepsilon)-v_t^g(\nu_\infty^\alpha)}{\varepsilon}} \leq Ce^{-\lambda t}\norme{\nabla g}{\infty}\nu_\infty^\alpha(|h|).\]
In particular,
\[\vabs{\int_{\RR^d}\nabla \S_tg(y)\cdot h(y)\,\nu_\infty^\alpha(\d y)}\leq Ce^{-\lambda t}\norme{\nabla g}{\infty} \int_{\RR^d} |h(y)| \nu_\infty^\alpha(\d y).\]
Since this inequality holds for every $h\in\H$, and in particular for every bounded measurable vector field, it follows that
\[|\nabla \S_tg(y)|\leq Ce^{-\lambda t}\norme{\nabla g}{\infty},\qquad \nu_\infty^\alpha\text{-almost everywhere}.\]
Since $\nabla \S_tg$ is continuous and $\nu_\infty^\alpha$ has full support, this bound holds everywhere on $\RR^d$. Taking the supremum over $y$ proves \ref{hyp:H3}, with $C_S=C$ and $\lambda_S=\lambda$.

\medskip
We next prove \textup{(ii)} $\;\Lra\;$ \textup{(i)}, using \cite[Theorem~2.5]{Cor25}, whose regularity and confinement assumptions follow from \ref{hyp:H1} and \ref{hyp:H2}. To do so, it suffices to establish the exponential decay of $\Omega_t$ in $\LLL(\H)$, see \cite[Proposition~2.6]{Cor25}. Let $h\in\H$ be such that $\norme{h}{2}\leq1$. Applying \eqref{eq:Omega-S-identification}, the Cauchy--Schwarz inequality, and \ref{hyp:H3} gives, for every $t\geq0$ and $x\in\RR^d$,
\[\vabs{\Omega_t(h)(x)}=\vabs{\int_{\RR^d}\nabla_y\S_tF^x(y)h(y)\,\nu_\infty^\alpha(\d y)}\leq\norme{\nabla_y\S_tF^x}{\infty}\norme{h}{2}\leq dC_Se^{-\lambda_St}\norme{\nabla_yF}{\infty},\]
where the last inequality follows from applying the gradient estimate in \ref{hyp:H3} componentwise to $F^x$. Squaring and integrating against the probability measure $\nu_\infty^\alpha$, we obtain $\norme{\Omega_t(h)}{2}\leq dC_Se^{-\lambda_St}\norme{\nabla_yF}{\infty}$. Consequently,
\[\sup_{t\geq0}\sup_{\norme{h}{2}\leq1}e^{\lambda_St}\norme{\Omega_t(h)}{2}\leq dC_S\norme{\nabla_yF}{\infty}<\infty.\]
The cited results therefore give \eqref{eq:non-linear-stability}, completing the proof.
\end{proof}

\section{Proof of Theorem~\ref{thm:main}}\label{sec:proof}

Throughout the proof, we use test functions in $\XXX$. This does not change the dual norm: convolution approximation preserves the Lipschitz constant and allows us to write, for every finite signed measure $\eta$ of zero total mass and finite first moment,
\begin{equation}\label{eq:dualnorm-X}
\norme{\eta}{\ast}=\sup_{\substack{g\in\XXX,\ g(0)=0\\\norme{\nabla g}{\infty}\leq1}}\vabs{\ps{g,\eta}}.
\end{equation}
We shall also apply the defining formula \eqref{eq:dualnorm} to auxiliary $C^1$ functions with bounded gradient, which need not belong to $\XXX$. For every such function $f$,
\begin{equation}\label{eq:dual-pairing-Lip}
\vabs{\ps{f,\eta}}\leq\norme{\nabla f}{\infty}\norme{\eta}{\ast}.
\end{equation}
No normalization is required in this inequality, since $\eta$ has zero total mass.

\medskip
Fix $T>0$. In order to prove \eqref{eq:main}, we have to control
\begin{equation}\label{eq:distance-dual}
\sup_{0\leq t\leq T}\norme{\mu_t^N-\nu_\infty^\alpha}{\ast}
=\sup_{0\leq t\leq T}\sup_{\substack{g\in\XXX,\ g(0)=0\\\norme{\nabla g}{\infty}\leq1}}\vabs{\ps{g,\nu_t^N}}
\end{equation}
The problem would be simpler if we only had to consider the measures on a compact set. Indeed, the class of normalized $1$-Lipschitz functions can then be approximated by a finite family of smooth functions. We therefore consider the ball $B_K:=\{x\in\RR^d:|x|\leq K\}$.

We first look at what happens in the tails, that is, at the part of the pairing in \eqref{eq:distance-dual} coming from $B_K^c$. The first observation is that this part can be controlled by the second moments. Indeed, for every function appearing in \eqref{eq:distance-dual}, the conditions $g(0)=0$ and $\norme{\nabla g}{\infty}\leq1$ imply that $|g(x)|\leq|x|$. Moreover, since $\nu_t^N=\mu_t^N-\nu_\infty^\alpha$, its total variation measure satisfies $|\nu_t^N|\leq\mu_t^N+\nu_\infty^\alpha$. Therefore, for every $K>0$,
\begin{equation}\label{eq:tail-second-moment}
\begin{aligned}
\int_{B_K^c}|g(x)|\,|\nu_t^N|(\d x)%
	&\leq\int_{B_K^c}|x|\,\parent{\mu_t^N+\nu_\infty^\alpha}(\d x)\\
	&\leq\frac1K\int_{B_K^c}|x|^2\,\parent{\mu_t^N+\nu_\infty^\alpha}(\d x)
	\leq\frac1K\parent{V_t^N+M_\infty},
\end{aligned}
\end{equation}
where $M_\infty$ is defined in \eqref{eq:Minfty}, while
\begin{equation} \label{eq:second-moment-tempo}
V_t^N:=\frac1N\sum_{i=1}^N|X_t^{i,N}|^2.
\end{equation}
Thus, the contribution outside $B_K$ can be controlled as soon as the empirical second moment $V_t^N$ remains bounded. 
This leads us naturally to consider, for a constant $M>0$ to be fixed below, the event
\begin{equation}\label{eq:Omega}
\EEE_T:=\left\{\sup_{0\leq t\leq T}V_t^N\leq2M\right\}.
\end{equation}
On this event, the estimate \eqref{eq:tail-second-moment} will allow us to make the contribution of $B_K^c$ sufficiently small by choosing $K$ appropriately. The choice of $K$ will be made in Step~2.

This gives the following natural decomposition of the proof:
\begin{itemize}
	\item[(1)] We first prove that $\EEE_T^c$, on which the empirical second moment becomes large, has exponentially small probability.
	\item[(2)] We then work on $\EEE_T$ and control the contribution coming from $B_K^c$. The estimate \eqref{eq:tail-second-moment} allows us to make this contribution sufficiently small by choosing $K$ sufficiently large.
	\item[(3)] Finally, still on $\EEE_T$, we restrict attention to $B_K$. The compactness of $B_K$ allows us to approximate the class of normalized $1$-Lipschitz test functions appearing in \eqref{eq:distance-dual} by a finite family of smooth functions. We then apply the mild equation to this finite family and estimate separately the propagated initial fluctuation, the nonlinear quadratic term, and the stochastic convolution.
\end{itemize}

\subsection{The empirical second moment remains bounded}

We start by proving that the event $\EEE_T^c$ has probability exponentially small in $N$, up to the factor $1+T$. Consequently, the empirical second moment remains bounded with high probability up to time horizons $T=e^{\gamma N}$ whenever $\gamma$ is smaller than the exponential rate obtained below.

\begin{lem}[Uniform moment confinement]\label{lem:moment}
There exist $M\geq M_{\rm init}$ and $C,c>0$ such that, for every $N\geq1$ and $T\geq0$,
\begin{equation}\label{eq:moment}
\PP\parent{\EEE_T^c}\leq C(1+T)e^{-cN}.
\end{equation}
\end{lem}

\begin{proof}
The idea is to control $V_t^N$ by a stochastic convolution and then apply the estimates established in Appendix~\ref{app:convolution} on stochastic convolutions. To this end, we first apply Itô's formula to $V_t^N$ and control separately the resulting terms.

For $|x|\ge R$, \ref{hyp:H1} with $y=0$ gives $2x\cdot b(x)\leq -2\beta|x|^2+2|b(0)||x| \leq -\beta|x|^2+\frac{|b(0)|^2}{\beta}$. By continuity on $B_R$, there is therefore $K_b<\infty$
\[2x\cdot b(x)\leq-\beta|x|^2+K_b,\qquad x\in\RR^d.\]
Since $F$ is bounded, the inequality $2uv\leq\frac{\beta}{2}u^2+\frac{2}{\beta}v^2$ gives, for every $N\geq1$, every $(y_1,\ldots,y_N)\in(\RR^d)^N$, and every $x\in\RR^d$,
\[2x\cdot\frac1N\sum_{j=1}^NF(x,y_j)\leq2|x|\norme{F}{\infty}\leq\frac{\beta}{2}|x|^2+\frac{2\norme{F}{\infty}^2}{\beta}.\]
Consequently, there exist constants $\kappa_V,K_V>0$, independent of $N$, such that, for every configuration $(x_1,\ldots,x_N)\in(\RR^d)^N$,
\begin{equation}\label{eq:collective-lyapunov}
\frac1N\sum_{i=1}^N\parent{2x_i\cdot b(x_i)+\frac2N\sum_{j=1}^Nx_i\cdot F(x_i,x_j)+d}\leq-\frac{\kappa_V}{N}\sum_{i=1}^N|x_i|^2+K_V.
\end{equation}
Applying Itô's formula to each $|X_t^{i,N}|^2$, averaging over $i$, and using the preceding estimate yield
\begin{equation}\label{eq:V-ineq}
\d V_t^N\leq(-\kappa_VV_t^N+K_V)\,\d t+\d\MMM_t^N,
\end{equation}
where $(\MMM_t^N)_{t\geq0}$ is the continuous martingale defined by
\begin{equation}
\MMM_t^N:=\frac2N\sum_{i=1}^N\int_0^tX_s^{i,N}\cdot\d B_s^i,\qquad \d\langle\MMM^N\rangle_t=\frac4NV_t^N\,\d t.
\end{equation}
Multiplying \eqref{eq:V-ineq} by the integrating factor $e^{\kappa_Vt}$ and integrating between $0$ and $t$ give
\[V_t^N\leq e^{-\kappa_Vt}V_0^N+K_V\int_0^te^{-\kappa_V(t-s)}\,\d s+\int_0^te^{-\kappa_V(t-s)}\,\d\MMM_s^N.\]
Using $V_0^N\leq M_{\rm init}$, we obtain
\begin{equation}\label{eq:V-mild}
V_t^N\leq M+R_t^N,\qquad M:=M_{\rm init}+\frac{K_V}{\kappa_V},\qquad R_t^N:=\int_0^te^{-\kappa_V(t-s)}\,\d\MMM_s^N.
\end{equation}

We would now like to apply Lemma~\ref{lem:stochastic-convolution} to the stochastic convolution $R^N$. However, by definition of $\MMM^N$, the squared norm of its kernel satisfies
\[\frac1N\sum_{i=1}^N\int_0^t\left|2e^{-\kappa_V(t-s)}X_s^{i,N}\right|^2\,\d s=4\int_0^te^{-2\kappa_V(t-s)}V_s^N\,\d s.\]
The right-hand side depends on the random process $V^N$ and therefore cannot be bounded deterministically as required by the lemma. Such a bound is available as long as $V_s^N$ does not exceed $2M$, so we consider the first hitting time
\[\tau_M:=\inf\{t\geq0:V_t^N\geq2M\}\]
and define the stopped convolution
\[\overline R_t^N:=\int_0^te^{-\kappa_V(t-s)}\1_{\{s\leq\tau_M\}}\,\d\MMM_s^N=\frac1N\sum_{i=1}^N\int_0^tH_i(t,s)\cdot\d B_s^i,\]
where $H_i(t,s):=2e^{-\kappa_V(t-s)}X_s^{i,N}\1_{\{s\leq\tau_M\}}$. We now check the hypotheses of Lemma~\ref{lem:stochastic-convolution}. These kernels are jointly measurable and predictable in $s$ for every fixed $t$.
First, \eqref{eq:fixed-kernel-condition} holds, since for every $t\geq0$,
\[\frac1N\sum_{i=1}^N \int_0^t |H_i(t,s)|^2\,\d s \;=\; 4\int_0^t e^{-2\kappa_V(t-s)} V_s^N\mathbf{1}_{\{s<\tau_M\}}\,\d s \leq 8 M \int_0^t e^{-2\kappa_V(t-s)}\,\d s \leq \frac{4M}{\kappa_V} =: a.\]
Next, fix $0\leq u\leq t\leq u+1$ and set $\delta:=t-u$.
The increment condition splits into the old-noise part and the contribution from $(u,t]$
\[\frac1N\sum_{i=1}^N \int_0^u |H_i(t,s)-H_i(u,s)|^2\,\d s + \frac1N\sum_{i=1}^N \int_u^t |H_i(t,s)|^2\,\d s.\]
For the old-noise part, since $e^{-\kappa_V(t-s)}-e^{-\kappa_V(u-s)} = e^{-\kappa_V(u-s)}(e^{-\kappa_V\delta}-1)$,
we have, using $1-e^{-\kappa_V\delta}\leq \kappa_V\delta$
\[\begin{aligned}
\frac1N\sum_{i=1}^N \int_0^u |H_i(t,s)-H_i(u,s)|^2\,\d s &= 4(1-e^{-\kappa_V\delta})^2 \int_0^u e^{-2\kappa_V(u-s)} V_s^N\mathbf{1}_{\{s<\tau_M\}}\,\d s  \\
&\leq 8M(1-e^{-\kappa_V\delta})^2 \int_0^u e^{-2\kappa_V(u-s)}\,\d s \leq 4M\kappa_V\,\delta^2.
\end{aligned}\]
For the new-noise part,
\[\frac1N\sum_{i=1}^N \int_u^t |H_i(t,s)|^2\,\d s = 4\int_u^t e^{-2\kappa_V(t-s)} V_s^N\mathbf{1}_{\{s<\tau_M\}}\,\d s \leq
8M\int_u^t e^{-2\kappa_V(t-s)}\,\d s \leq 8M\,\delta.\]
Since $\delta\leq1$, the total increment variance is bounded by $4M\kappa_V\,\delta^2+8M\,\delta \leq 4M(\kappa_V+2)\delta$. Therefore, \eqref{eq:increment-kernel-condition} holds with
\[\rho:=\frac12,\qquad b:=4M(\kappa_V+2).\]
Since $M$, $a$, and $b$ are deterministic, applying Lemma~\ref{lem:stochastic-convolution} to $\overline R^N$ with $x=M/2$ yields
\[\PP\parent{\sup_{0\leq t\leq T}|\overline R_t^N|>\frac M2}\leq C(1+T)e^{-cN},\]
where one may take $C=C_{1/2}$ and $c=c_{1/2}M^2/(4(a+b))$.

Finally, it remains to relate the preceding estimate to $\PP(\EEE_T^c)$. By the definitions of $\EEE_T$ and $\tau_M$, we first have $\EEE_T^c \subset \{\tau_M\leq T\}$. 
Moreover, on the event $\event{\tau_M\leq T}$, the preceding estimate gives $\overline R_{\tau_M}^N\geq M>M/2$. Consequently,
\[\event{\tau_M\leq T}\subset\event{\sup_{0\leq t\leq T}|\overline R_t^N|>\frac M2}.\]
To obtain the first inequality, the continuity of $V^N$ gives $V_{\tau_M}^N=2M$, while $R_{\tau_M}^N=\overline R_{\tau_M}^N$. It follows from \eqref{eq:V-mild} that
\[2M=V_{\tau_M}^N\leq M+\overline R_{\tau_M}^N,\]
which gives $\overline R_{\tau_M}^N\geq M$. Combining these inclusions with the stochastic convolution estimate yields
\[\PP\parent{\EEE_T^c}\leq\PP(\tau_M\leq T)\leq C(1+T)e^{-cN}.\]
We have therefore proved \eqref{eq:moment}.
\end{proof}

\subsection{Analysis under bounded empirical second moments}

It remains to control \eqref{eq:distance-dual} on the event $\EEE_T$, on which the empirical second moment remains bounded throughout $[0,T]$. We first briefly show that the contribution coming from $B_K^c$ can be made arbitrarily small, and then reduce the remainder of the analysis to the compact ball $B_K$.

We begin by controlling the tails on $B_K^c$. Set $C_M:=2M+M_\infty$. For every $\delta>0$, choose $K:=C_M/\delta$. It follows from \eqref{eq:tail-second-moment} that, on $\EEE_T$,
\begin{equation}\label{eq:uniform-tail}
\sup_{0\leq t\leq T}\sup_{g \in \XXX, \substack{g(0)=0\\\norme{\nabla g}{\infty}\leq1}}\vabs{\int_{B_K^c}g(x)\,\nu_t^N(\d x)} \leq \sup_{0\leq t\leq T}\sup_{g \in \XXX, \substack{g(0)=0\\ \norme{\nabla g}{\infty}\leq1}}\int_{B_K^c}|g(x)|\,|\nu_t^N|(\d x)\leq\frac{C_M}{K}=\delta.
\end{equation}
Thus, uniformly in time and over all normalized $1$-Lipschitz test functions, the contribution over $B_K^c$ is at most $\delta$.

It remains to study the contribution inside $B_K$. The compactness of this ball allows us to approximate the class of normalized $1$-Lipschitz test functions appearing in \eqref{eq:distance-dual} by a finite family of smooth functions. These functions will be defined on all of $\RR^d$, so that the mild equation can still be applied on the whole space. The argument is organized as follows. In Step~1, we construct such a finite family. In Step~2, we derive the mild equation for each function in this family. In Step~3, we estimate separately the propagated initial fluctuation, the nonlinear quadratic term, and the stochastic convolution appearing in this equation.

\paragraph{Step 1: Reduction to a compact set and the \texorpdfstring{$\delta$}{delta}-net}
We first restrict the normalized $1$-Lipschitz test functions to $B_K$ to exploit its compactness. We then extend and regularize the approximating functions to obtain smooth functions on $\RR^d$. The approximation is only required on $B_K$.

\begin{lem}[Finite smooth net]\label{lem:net}
For every $K,\delta>0$, there exist an integer $m=m(K,\delta)\geq1$ and a finite family of functions $\GGG_{K,\delta}:=\{g_1,\ldots,g_m\}$, such that
\begin{itempuce}
\item $g_1,\ldots,g_m\in\DDD$;
\item for every $k\in\{1,\ldots,m\}$, $g_k(0)=0$ and $\norme{\nabla g_k}{\infty}\leq1$;
\item every globally Lipschitz function $g$ satisfying $g(0)=0$ and $\norme{\nabla g}{\infty}\leq1$ satisfies
\begin{equation}\label{eq:net-approx}
\min_{1\leq k\leq m}\norme{g-g_k}{L^\infty(B_K)}\leq\delta.
\end{equation}
\end{itempuce}
\end{lem}

The proof is given in Appendix~\ref{app:finite-net}.

Recall that $C_M:=2M+M_\infty$. Fix $\delta>0$ and set $K:=C_M/\delta$. By Lemma~\ref{lem:net}, this choice yields a finite family $\GGG_\delta:=\{g_1,\ldots,g_m\}\subset\DDD$.

\begin{lem}[Finite-dimensional reduction]\label{lem:finite-reduction}
On the event \(\EEE_T\), one has
\begin{equation}\label{eq:finite-reduction}
\sup_{0\leq t\leq T}\norme{\nu_t^N}{\ast} \leq\, \sup_{0\leq t\leq T} \max_{1 \leq k \leq m} \vabs{\ps{g_k,\nu_t^N}} + 4\delta.
\end{equation}
\end{lem}

\begin{proof}
Let $t\in[0,T]$ and fix an admissible test function $g$, namely $g(0)=0$ and $\norme{\nabla g}{\infty}\leq1$. By Lemma~\ref{lem:net}, we can choose $g_k\in\GGG_\delta$ such that $\norme{g-g_k}{L^\infty(B_K)}\leq\delta$. By the triangle inequality,
\[\vabs{\ps{g,\nu_t^N}}\leq\vabs{\ps{g_k,\nu_t^N}}+\vabs{\ps{g-g_k,\nu_t^N}}.\]
It is therefore sufficient to show that $\vabs{\ps{g-g_k,\nu_t^N}}\leq4\delta$. To this end, we split the last pairing over $B_K$ and $B_K^c$
\[\ps{g-g_k,\nu_t^N} := \int_{\RR^d} (g-g_k) \,\d\nu_t^N = \int_{B_K}(g-g_k)\,\d\nu_t^N + \int_{B_K^c}(g-g_k)\,\d\nu_t^N.\]
On $B_K^c$, both $g$ and $g_k$ are globally $1$-Lipschitz and vanish at the origin. Hence, $|g(x)-g_k(x)|\leq2|x|$. Since $|\nu_t^N|\leq\mu_t^N+\nu_\infty^\alpha$, it follows from \eqref{eq:uniform-tail} that, on $\EEE_T$,
\[\left|\int_{B_K^c}(g-g_k)\,\d\nu_t^N\right|\leq2\int_{B_K^c}|x|\,\parent{\mu_t^N+\nu_\infty^\alpha}(\d x)\leq\frac{2C_M}{K}=2\delta.\]
On $B_K$, since $\mu_t^N$ and $\nu_\infty^\alpha$ are probability measures, $|\nu_t^N|(B_K)\leq2$. Therefore,
\[\left|\int_{B_K}(g-g_k)\,\d\nu_t^N\right|\leq\norme{g-g_k}{L^\infty(B_K)}|\nu_t^N|(B_K)\leq2\delta.\]
Combining the two estimates gives
\[\vabs{\ps{g,\nu_t^N}}\leq\vabs{\ps{g_k,\nu_t^N}}+4\delta.\]
Taking the supremum over all admissible test functions $g$ and then over $t\in[0,T]$ yields \eqref{eq:finite-reduction}.
\end{proof}

\paragraph{Step 2: The mild equation}
Lemma~\ref{lem:finite-reduction} reduces the problem to the study of the finite family $\GGG_\delta$ and therefore to the control of the term
\[\sup_{0\leq t\leq T} \max_{1 \leq k \leq m} \vabs{\ps{g_k,\nu_t^N}}.\]
To control this term, we follow the mild-formulation approach of \cite[Proposition~A.1]{BGP14} and establish the following representation of the empirical fluctuations in terms of the linearized semigroup.

\begin{prop}[Mild equation]\label{prop:mild}
For every $N\geq1$, the following identity holds almost surely, simultaneously for every $t\geq0$ and every $g \in \GGG_\delta$
\begin{equation}\label{eq:me}
\ps{g,\nu_t^N}=\ps{\S_t g,\nu_0^N}+\int_0^t\iint_{\RR^d\times\RR^d}\nabla \S_{t-s}g(x)\cdot F(x,y)\,\nu_s^N(\d x)\nu_s^N(\d y)\,\d s+z_t^N(g),
\end{equation}
where
\begin{equation}\label{eq:ztN}
z_t^N(g):=\frac1N\sum_{i=1}^N\int_0^t\nabla \S_{t-s}g(X_s^{i,N})\cdot\d B_s^i.
\end{equation}
\end{prop}

\begin{proof}
Fix $g\in\GGG_\delta$ and $t>0$, and set $g_s:=\S_{t-s}g$ for $s\in[0,t]$. 

We first compute $\ps{g,\mu_t^N}$. Since the semigroup satisfies $\partial_r\S_rg=\AAA \S_rg$, we have $\partial_sg_s=-\AAA g_s$. Applying It\^o's formula to the process $g_s(X_s^{i,N})$, we obtain
\[\begin{aligned}
\d g_s(X_s^{i,N})={}&\partial_sg_s(X_s^{i,N})\,\d s+\nabla g_s(X_s^{i,N})\cdot\d X_s^{i,N}+\frac12\Delta g_s(X_s^{i,N})\,\d s\\
={}&\parent{-\AAA g_s(X_s^{i,N})+\nabla g_s(X_s^{i,N})\cdot\parent{b(X_s^{i,N})+\Phi(\mu_s^N)(X_s^{i,N})}+\frac12\Delta g_s(X_s^{i,N})}\d s\\
&+\nabla g_s(X_s^{i,N})\cdot\d B_s^i.
\end{aligned}\]
By the definition of $\LLL_\alpha$, $\nabla g_s\cdot b+\frac12\Delta g_s=\LLL_\alpha g_s-\nabla g_s\cdot\alpha$, while $\AAA=\LLL_\alpha+\KKK$ gives $-\AAA g_s=-\LLL_\alpha g_s-\KKK g_s$. Consequently, $-\AAA g_s+\nabla g_s\cdot b+\frac12\Delta g_s=-\nabla g_s\cdot\alpha-\KKK g_s$, and therefore
\[\d g_s(X_s^{i,N})=\nabla g_s(X_s^{i,N})\cdot\parent{\Phi(\mu_s^N)(X_s^{i,N})-\alpha(X_s^{i,N})}\,\d s-\KKK g_s(X_s^{i,N})\,\d s+\nabla g_s(X_s^{i,N})\cdot\d B_s^i.\]
Since $\alpha=\Phi(\nu_\infty^\alpha)$ and $\KKK g_s(y)=\int_{\RR^d}F(x,y)\cdot\nabla g_s(x)\,\nu_\infty^\alpha(\d x)$, the preceding identity becomes
\[\begin{aligned}
\d g_s(X_s^{i,N})={}&\int_{\RR^d}\nabla g_s(X_s^{i,N})\cdot F(X_s^{i,N},y)\,\mu_s^N(\d y)\,\d s\\
	&-\int_{\RR^d}\nabla g_s(X_s^{i,N})\cdot F(X_s^{i,N},y)\,\nu_\infty^\alpha(\d y)\,\d s\\
	&-\int_{\RR^d}\nabla g_s(x)\cdot F(x,X_s^{i,N})\,\nu_\infty^\alpha(\d x)\,\d s+\nabla g_s(X_s^{i,N})\cdot\d B_s^i.
\end{aligned}\]
Integrating this stochastic differential over $[0,t]$, summing the resulting identities over $i=1,\ldots,N$, and dividing by $N$, we arrive at
\[\begin{aligned}
\ps{g,\mu_t^N}-\ps{\S_tg,\mu_0^N}={}&\int_0^t\iint_{\RR^d\times\RR^d}\nabla g_s(x)\cdot F(x,y)\,\mu_s^N(\d x)\mu_s^N(\d y)\,\d s\\
	&-\int_0^t\iint_{\RR^d\times\RR^d}\nabla g_s(x)\cdot F(x,y)\,\mu_s^N(\d x)\nu_\infty^\alpha(\d y)\,\d s\\
	&-\int_0^t\iint_{\RR^d\times\RR^d}\nabla g_s(x)\cdot F(x,y)\,\nu_\infty^\alpha(\d x)\mu_s^N(\d y)\,\d s+z_t^N(g).
\end{aligned}\]
We now turn to $\ps{g,\nu_\infty^\alpha}$. Let
$(X_s^{\nu_\infty^\alpha})_{s\geq0}$ denote the solution of the McKean--Vlasov
equation~\eqref{eq:mvl} initialized with law $\nu_\infty^\alpha$. Since
$\nu_\infty^\alpha$ is invariant, the interaction field remains equal to
$\Phi(\nu_\infty^\alpha)=\alpha$, and the process solves equation~\eqref{eq:mvll}. Applying It\^o's formula to $g_s(X_s^{\nu_\infty^\alpha})$ and using $\AAA=\LLL_\alpha+\KKK$, we obtain
\[\begin{aligned}
\d g_s(X_s^{\nu_\infty^\alpha}) &=\parent{-\AAA g_s+\LLL_\alpha g_s}(X_s^{\nu_\infty^\alpha})\,\d s+\nabla g_s(X_s^{\nu_\infty^\alpha}) \cdot\d B_s\\
	&=-\KKK g_s(X_s^{\nu_\infty^\alpha})\,\d s+\nabla g_s(X_s^{\nu_\infty^\alpha})\cdot\d B_s.
\end{aligned}\]
Integrating between $0$ and $t$ gives
\[g(X_t^{\nu_\infty^\alpha})-\S_tg(X_0^{\nu_\infty^\alpha})=-\int_0^t\KKK g_s(X_s^{\nu_\infty^\alpha})\,\d s+\int_0^t\nabla g_s(X_s^{\nu_\infty^\alpha})\cdot\d B_s.\]
Taking expectations, using $\Law(X_s^{\nu_\infty^\alpha})=\nu_\infty^\alpha$ and the fact that the stochastic integral has zero expectation, yields
\[\ps{g,\nu_\infty^\alpha} - \ps{\S_tg,\nu_\infty^\alpha} = -\int_0^t\ps{\KKK g_s,\nu_\infty^\alpha}\,\d s = -\int_0^t\iint_{\RR^d\times\RR^d}\nabla g_s(x)\cdot F(x,y)\,\nu_\infty^\alpha(\d x)\nu_\infty^\alpha(\d y)\,\d s.\]

Finally, recalling that $\nu_s^N=\mu_s^N-\nu_\infty^\alpha$, we subtract the two preceding integral equations and obtain
\[\begin{aligned}
\ps{g,\nu_t^N}={}&\ps{\S_tg,\nu_0^N}+\int_0^t\Bigg[\iint_{\RR^d\times\RR^d}\nabla g_s(x)\cdot F(x,y)\,\mu_s^N(\d x)\mu_s^N(\d y)\\
	&-\iint_{\RR^d\times\RR^d}\nabla g_s(x)\cdot F(x,y)\,\mu_s^N(\d x)\nu_\infty^\alpha(\d y)\\
	&-\iint_{\RR^d\times\RR^d}\nabla g_s(x)\cdot F(x,y)\,\nu_\infty^\alpha(\d x)\mu_s^N(\d y)\\
	&+\iint_{\RR^d\times\RR^d}\nabla g_s(x)\cdot F(x,y)\,\nu_\infty^\alpha(\d x)\nu_\infty^\alpha(\d y)\Bigg]\,\d s+z_t^N(g).
\end{aligned}\]
Using the algebraic identity
\[\begin{aligned}
\iint H(x,y)(\mu-\nu)(\d x)(\mu-\nu)(\d y)={}&\iint H(x,y)\mu(\d x)\mu(\d y)-\iint H(x,y)\mu(\d x)\nu(\d y)\\
	&-\iint H(x,y)\nu(\d x)\mu(\d y)+\iint H(x,y)\nu(\d x)\nu(\d y),
\end{aligned}\]
with $H(x,y):=\nabla \S_{t-s}g(x)\cdot F(x,y)$, $\mu=\mu_s^N$ and $\nu=\nu_\infty^\alpha$, we recover exactly \eqref{eq:me}.
\end{proof}

\paragraph{Step 3: Estimating the terms of the mild equation}
It remains to estimate each term on the right-hand side of the mild equation~\eqref{eq:me}, uniformly over $[0,T]$. Using the estimates for the semigroup $(\S_t)_{t\geq0}$, we obtain the following result.

\begin{lem}[Estimates of the terms in the mild equation]\label{lem:mild-estimates} 
On the event $\EEE_T$, for every $0\leq t\leq T$,
\begin{equation}\label{eq:main-ineq-before-exit}
\max_{1 \leq k \leq m} \vabs{\ps{g_k,\nu_t^N}} \leq C_S\norme{\nu_0^N}{\ast}+K_0\sup_{0\leq s\leq t}\norme{\nu_s^N}{\ast}^2+\max_{1 \leq k \leq m}|z_t^N(g_k)|,
\end{equation}
where $K_0:=L_F\parent{C_H\sqrt{\frac{\pi}{\lambda_H}}+\frac{C_S}{\lambda_S}}$.

Moreover, there exists $C_m<\infty$, depending on the finite family $\GGG_\delta$ such that, for every $T,x>0$,
\begin{equation}\label{eq:mild-noise-max}
\PP\parent{\sup_{0\leq t\leq T} \, \max_{1\leq k\leq m}\vabs{z_t^N(g_k)}>x}\leq C_{1/2}m(1+T)\exp\parent{-c_{1/2}\frac{Nx^2}{C_m^2}}.
\end{equation}
\end{lem}

\begin{proof}
Fix $0\leq t\leq T$ and $1\leq k\leq m$.

For the initial term, since $\nu_0^N$ has zero total mass, the definition of the dual norm, assumption \ref{hyp:H3}, and the inequality $\norme{\nabla g_k}{\infty}\leq1$ give
\[\vabs{\ps{\S_tg_k,\nu_0^N}}\leq\norme{\nabla \S_tg_k}{\infty}\norme{\nu_0^N}{\ast}\leq C_Se^{-\lambda_St}\norme{\nu_0^N}{\ast}\leq C_S\norme{\nu_0^N}{\ast}.\]

For the quadratic term, for $r>0$, set $G_r(x,y):=\nabla \S_rg_k(x)\cdot F(x,y)$. Then
\[\int_0^t\iint_{\RR^d\times\RR^d}\nabla \S_{t-s}g_k(x)\cdot F(x,y)\,\nu_s^N(\d x)\nu_s^N(\d y)\,\d s = \int_0^t \iint_{\RR^d\times\RR^d} G_{t-s}(x,y)\,\nu_s^N(\d x)\nu_s^N(\d y)\,\d s.\]
We first apply the definition of $\norme{\cdot}{\ast}$ to a function of the variable $x$, and then to a function of the variable $y$. Note that $\nu_s^N$ has zero total mass, so that no normalization at the origin is needed. For the first application, fix $y,y'\in\RR^d$ and consider the globally Lipschitz function $x\lms G_r(x,y)-G_r(x,y')$ of Lipschitz seminorm bounded by $\norme{\nabla_xG_r(\cdot,y)-\nabla_xG_r(\cdot,y')}{\infty}$. Hence,
\[\vabs{\ps{G_r(\cdot,y)-G_r(\cdot,y'),\nu_s^N}} \leq \norme{\nabla_xG_r(\cdot,y)-\nabla_xG_r(\cdot,y')}{\infty}\norme{\nu_s^N}{\ast}.\]
Moreover, the fundamental theorem of calculus gives, for every $x\in\RR^d$,
\begin{align*}
	\vabs{\nabla_xG_r(x,y)-\nabla_xG_r(x,y')} &=
\vabs{\int_0^1\nabla_y\nabla_xG_r\parent{x,y'+t(y-y')}\parent{y-y'}\,\d t} \\
& \leq \norme{\nabla_x\nabla_yG_r}{\infty}|y-y'|,
\end{align*}
therefore $\norme{\nabla_xG_r(\cdot,y)-\nabla_xG_r(\cdot,y')}{\infty} \leq
\norme{\nabla_x\nabla_yG_r}{\infty}|y-y'|$. Hence
\[\vabs{\ps{G_r(\cdot,y)-G_r(\cdot,y'),\nu_s^N}}\leq\norme{\nabla_x\nabla_yG_r}{\infty}\norme{\nu_s^N}{\ast}|y-y'|.\]
This shows that the function $y\lms\ps{G_r(\cdot,y),\nu_s^N}$ is globally Lipschitz, with Lipschitz seminorm bounded by $\norme{\nabla_x\nabla_yG_r}{\infty}\norme{\nu_s^N}{\ast}$. Applying the definition of $\norme{\cdot}{\ast}$ to this function of $y$, we obtain
\[\vabs{\iint_{\RR^d\times\RR^d}G_r(x,y)\,\nu_s^N(\d x)\nu_s^N(\d y)} = \vabs{\int_{\RR^d}\ps{G_r(\cdot,y),\nu_s^N}\,\nu_s^N(\d y)} \leq \norme{\nabla_x\nabla_yG_r}{\infty}\norme{\nu_s^N}{\ast}^2.\]
Taking $r=t-s$ gives
\begin{equation}\label{eq:lem-ineg-on-G}
\vabs{\iint_{\RR^d\times\RR^d}G_{t-s}(x,y)\,\nu_s^N(\d x)\nu_s^N(\d y)}\leq\norme{\nabla_x\nabla_yG_{t-s}}{\infty}\norme{\nu_s^N}{\ast}^2.
\end{equation}
Moreover, $\nabla_x\nabla_yG_r(x,y)=\nabla^2\S_rg_k(x)\cdot \nabla_yF(x,y)+\nabla \S_rg_k(x)\cdot \nabla_x\nabla_yF(x,y)$. Therefore, assumption \ref{hyp:H1}, followed by \eqref{eq:hessian} and \eqref{eq:linear-stability}, gives
\begin{equation}\label{eq:lem-ineg-on-G2}
\norme{\nabla_x\nabla_yG_r}{\infty} \leq \parent{\norme{\nabla^2 \S_r g_k}{\infty}+\norme{\nabla \S_r g_k}{\infty}} L_F \leq L_F \parent{C_Hr^{-1/2}e^{-\lambda_Hr}+C_Se^{-\lambda_Sr}},
\end{equation}
where we used $\norme{\nabla g_k}{\infty}\leq1$. Combining \eqref{eq:lem-ineg-on-G} and \eqref{eq:lem-ineg-on-G2}, and then integrating between $0$ and $t$, we obtain
\[\begin{aligned}
& \vabs{\int_0^t\iint_{\RR^d\times\RR^d}\nabla \S_{t-s}g_k(x)\cdot F(x,y)\,\nu_s^N(\d x)\nu_s^N(\d y)\,\d s}\\
\leq & L_F\int_0^t\parent{C_H(t-s)^{-1/2}e^{-\lambda_H(t-s)}+C_Se^{-\lambda_S(t-s)}}\norme{\nu_s^N}{\ast}^2\,\d s\\
\leq & K_0\sup_{0\leq s\leq t}\norme{\nu_s^N}{\ast}^2,
\end{aligned}\]
where the last inequality follows by taking the supremum of $\norme{\nu_s^N}{\ast}^2$ over $s\in[0,t]$, using the change of variable $r=t-s$, and extending the integral to $[0,\infty)$, since
\[L_F\int_0^\infty\parent{C_Hr^{-1/2}e^{-\lambda_Hr}+C_Se^{-\lambda_Sr}}\,\d r=L_F\parent{C_H\sqrt{\frac{\pi}{\lambda_H}}+\frac{C_S}{\lambda_S}}=K_0.\]
Combining the two estimates on the initial and the quadratic terms with the mild equation and then taking the maximum over $1\leq k\leq m$ proves \eqref{eq:main-ineq-before-exit}.

The last term is a stochastic convolution in the sense of Appendix~\ref{app:convolution}, with $H_i(t,s):=\nabla \S_{t-s}g_k(X_s^{i,N})$
\[z_t^N(g_k)=\frac1N\sum_{i=1}^N\int_0^t\nabla \S_{t-s}g_k(X_s^{i,N})\cdot\d B_s^i=\frac1N\sum_{i=1}^N\int_0^tH_i(t,s)\cdot\d B_s^i.\]
We therefore check the two bounds required by Lemma~\ref{lem:stochastic-convolution}. The first bound follows from \eqref{eq:linear-stability} and $\norme{\nabla g_k}{\infty}\leq1$
\[\frac1N\sum_{i=1}^N\int_0^t|H_i(t,s)|^2\,\d s=\frac1N\sum_{i=1}^N\int_0^t\vabs{\nabla \S_{t-s}g_k(X_s^{i,N})}^2\,\d s\leq C_S^2\int_0^te^{-2\lambda_S(t-s)}\,\d s\leq\frac{C_S^2}{2\lambda_S}=:a.\]
For the second bound, fix $0\leq u\leq t\leq u+1$. Applying \eqref{eq:linear-stability} once more and using $e^{-2\lambda_S(t-s)}\leq1$, we obtain
\[\frac1N\sum_{i=1}^N\int_u^t|H_i(t,s)|^2\,\d s=\frac1N\sum_{i=1}^N\int_u^t\vabs{\nabla \S_{t-s}g_k(X_s^{i,N})}^2\,\d s\leq C_S^2\int_u^te^{-2\lambda_S(t-s)}\,\d s\leq C_S^2(t-u).\]
Moreover, using \eqref{eq:terminal-increment} with $r=u-s$ and $h=t-u$, we obtain
\[\begin{aligned}\frac1N\sum_{i=1}^N\int_0^u|H_i(t,s)-H_i(u,s)|^2\,\d s&=\frac1N\sum_{i=1}^N\int_0^u\vabs{\nabla \S_{t-s}g_k(X_s^{i,N})-\nabla \S_{u-s}g_k(X_s^{i,N})}^2\,\d s\\&\leq C_S^2C_{g_k}^2|t-u|^2\int_0^ue^{-2\lambda_S(u-s)}\,\d s\\&\leq\frac{C_S^2C_{g_k}^2}{2\lambda_S}|t-u|^2.\end{aligned}\]
Since $|t-u|\leq1$, we have $|t-u|^2\leq|t-u|$. Adding the preceding estimates therefore gives
\[\frac1N\sum_{i=1}^N\parent{\int_0^u|H_i(t,s)-H_i(u,s)|^2\,\d s+\int_u^t|H_i(t,s)|^2\,\d s}\leq C_S^2\parent{1+\frac{C_{g_k}^2}{2\lambda_S}}|t-u|.\]
Thus, \eqref{eq:increment-kernel-condition} holds with $\rho=1/2$ and $b:=\max_{1\leq k\leq m}C_S^2\parent{1+\frac{C_{g_k}^2}{2\lambda_S}}$. Applying Lemma~\ref{lem:stochastic-convolution} to $z^N(g_k)$, we obtain
\[\PP\parent{\sup_{0\leq t\leq T}\vabs{z_t^N(g_k)}>x}\leq C_{1/2}(1+T)\exp\parent{-c_{1/2}\frac{Nx^2}{a+b}}.\]
Hence, it remains to set $C_m^2:=a+b$ and take a union bound over $1\leq k\leq m$ to prove \eqref{eq:mild-noise-max}.
\end{proof}

\subsection{Proof of Theorem~\ref{thm:main} and Corollary \ref{cor:mean-exit-time-lower-bound}}

To prove Theorem~\ref{thm:main}, it remains to relate the event $\event{\sup_{0\leq t\leq T}\norme{\nu_t^N}{\ast}\geq a}$ to the preceding estimates. We first split its probability according to whether $\EEE_T$ occurs:
\[\begin{aligned}
\PP\parent{\sup_{0\leq t\leq T}\norme{\nu_t^N}{\ast}\geq a}&=\,\PP\parent{\event{\sup_{0\leq t\leq T}\norme{\nu_t^N}{\ast}\geq a}\cap\EEE_T^c}+\PP\parent{\event{\sup_{0\leq t\leq T}\norme{\nu_t^N}{\ast}\geq a}\cap\EEE_T}\\
	&\leq\,\PP(\EEE_T^c)+\PP\parent{\event{\sup_{0\leq t\leq T}\norme{\nu_t^N}{\ast}\geq a}\cap\EEE_T}.
\end{aligned}\]
The first term has already been estimated in \eqref{eq:moment}. To control the second one, we show that, on $\EEE_T$, reaching the distance $a$ requires at least one of the stochastic convolutions to become sufficiently large. Indeed, combining the finite-dimensional reduction with the estimates of the terms in the mild equation, and taking the supremum over $[0,T]$, gives, on $\EEE_T$,
\begin{equation}\label{eq:main-ineq-before-exit2}
\begin{aligned}
\sup_{0\leq t\leq T}\norme{\nu_t^N}{\ast}&\leq\sup_{0\leq t\leq T}\max_{1\leq k\leq m}\vabs{\ps{g_k,\nu_t^N}}+4\delta\\
&\leq C_S\norme{\nu_0^N}{\ast}+K_0\sup_{0\leq t\leq T}\norme{\nu_t^N}{\ast}^2+\sup_{0\leq t\leq T}\max_{1\leq k\leq m}\vabs{z_t^N(g_k)}+4\delta.
\end{aligned}
\end{equation}

\begin{proof}[Proof of Theorem~\ref{thm:main}]
We first choose $a_0>0$ small enough that $K_0a_0\leq1/2$. Fix $a\in(0,a_0]$ and $T\geq0$. We take $\delta:=a/32$ and $K:=C_M/\delta$, and denote by $\GGG_a:=\{g_1,\ldots,g_m\}$ the corresponding finite family of test functions.

Using the initial bound \eqref{eq:initial-conditions}, we shall prove the following inclusion, almost surely:
\[\event{\sup_{0\leq t\leq T}\norme{\nu_t^N}{\ast}\geq a}\cap\EEE_T\,\subset\,\event{\sup_{0\leq t\leq T}\max_{1\leq k\leq m}\vabs{z_t^N(g_k)}>\frac a8}.\]
This will allow us to apply \eqref{eq:mild-noise-max}. We work directly on the event on the left-hand side. For $0\leq t\leq T$, set
\[f(t):=\sup_{0\leq s\leq t}\norme{\nu_s^N}{\ast}.\]
Since $\EEE_T\subset\EEE_t$ for $t\leq T$, we may apply \eqref{eq:main-ineq-before-exit2} with horizon $t$. Using the initial bound $C_S\norme{\nu_0^N}{\ast}\leq a/8$ together with our choice $4\delta=a/8$ gives
\[f(t)\leq\frac a4+K_0f(t)^2+\sup_{0\leq s\leq t}\max_{1\leq k\leq m}\vabs{z_s^N(g_k)},\qquad 0\leq t\leq T.\]
We now use a continuity argument similar to that underlying \cite[Lemma~B.4]{Cop22}. Since the particle trajectories are continuous, $f$ is continuous. Moreover, $f(0)\leq a/(8C_S)<a$, while $f(T)\geq a$ by the definition of the event considered above. The intermediate value theorem therefore gives some $t_\ast\in[0,T]$ such that $f(t_\ast)=a$. Applying the preceding estimate at $t=t_\ast$, and using $K_0a\leq1/2$ and $t_\ast\leq T$, yields
\[a\leq\frac a4+K_0a^2+\sup_{0\leq s\leq t_\ast}\max_{1\leq k\leq m}\vabs{z_s^N(g_k)}\leq\frac{3a}{4}+\sup_{0\leq s\leq T}\max_{1\leq k\leq m}\vabs{z_s^N(g_k)}.\]
Consequently,
\[\sup_{0\leq t\leq T}\max_{1\leq k\leq m}\vabs{z_t^N(g_k)}\geq\frac a4>\frac a8,\]
which proves the desired inclusion.

It remains to apply the probability estimates established above. Combining this inclusion with \eqref{eq:moment} and \eqref{eq:mild-noise-max}, the latter with $x=a/8$, gives
\[\begin{aligned}\PP\parent{\sup_{0\leq t\leq T}\norme{\nu_t^N}{\ast}\geq a}&\leq\PP(\EEE_T^c)+\PP\parent{\sup_{0\leq t\leq T}\max_{1\leq k\leq m}\vabs{z_t^N(g_k)}>\frac a8}\\&\leq C(1+T)e^{-cN}+C_{1/2}m(1+T)\exp\parent{-\frac{c_{1/2}a^2}{64C_m^2}N}.\end{aligned}\]
Setting $C_a:=C+C_{1/2}m$ and $c_a:=\min\left\{c,\frac{c_{1/2}a^2}{64C_m^2}\right\}>0$, we conclude that
\[\PP\parent{\sup_{0\leq t\leq T}\norme{\nu_t^N}{\ast}\geq a}\leq C_a(1+T)e^{-c_aN}.\]
This proves \eqref{eq:main}.
\end{proof}

\begin{proof}[Proof of Corollary~\ref{cor:mean-exit-time-lower-bound}]
By Markov's inequality and Theorem~\ref{thm:main}, for every $T>0$,
\[\EE(\tau_a^N)\geq T\PP(\tau_a^N>T)\geq T\parent{1-C_a(1+T)e^{-c_aN}}.\]
For sufficiently large $N$, choose $T_N>0$ such that $C_a(1+T_N)e^{-c_aN}=1/2$, namely
\[T_N=\frac{e^{c_aN}}{2C_a}-1.\]
Consequently, for $N$ sufficiently large that $e^{c_aN}\geq4C_a$,
\[\EE(\tau_a^N)\geq\frac{T_N}{2}=\frac{e^{c_aN}}{4C_a}-\frac12\geq\frac{e^{c_aN}}{8C_a}.\]
\end{proof}

\section{Application to an interacting Ornstein--Uhlenbeck model}\label{sec:ou}

We now apply Theorem~\ref{thm:main} to the Ornstein--Uhlenbeck--cosine model considered by Cormier \cite[Section~2.2]{Cor25}, with Brownian coefficient $1$ instead of $\sqrt{2}$. We first recall its invariant laws and their stability criterion, and then deduce confinement near each stable equilibrium.

Consider the one-dimensional particle system
\begin{equation}\label{eq:OU-particles}
\d X_t^{i,N}=\parent{-X_t^{i,N}+\frac{J}{N}\sum_{j=1}^N\cos(X_t^{j,N})}\d t+\d B_t^i,\qquad 1\leq i\leq N,\qquad J\geq0.
\end{equation}
This is the model \eqref{eq:spl} with $b(x)=-x$ and $F(x,y)=J\cos y$. The limiting McKean--Vlasov equation is
\begin{equation}\label{eq:OU-cosine-MV}
\d X_t=\parent{-X_t+J\EE[\cos(X_t)]}\d t+\d B_t.
\end{equation}

\paragraph{Invariant laws}
For a constant interaction field $\alpha\in\RR$, the diffusion \eqref{eq:mvll} is an Ornstein--Uhlenbeck process whose unique invariant law is $\nu_\infty^\alpha=\NNN(\alpha,1/2)$. Since
\[\int_\RR\cos x\,\NNN(\alpha,1/2)(\d x)=e^{-1/4}\cos\alpha,\]
the characterization \eqref{eq:general-fixed-point} shows that the invariant laws of \eqref{eq:OU-cosine-MV} are precisely the measures $\NNN(\alpha,1/2)$ for which
\begin{equation}\label{eq:OU-fixed-point}
\alpha=Je^{-1/4}\cos\alpha.
\end{equation}

\paragraph{Stability}
Fix a solution $\alpha$ of \eqref{eq:OU-fixed-point}. Applying the stability criterion of \cite[Theorem~2.5 and Section~2.2]{Cor25} with our Brownian coefficient, including the necessity discussed there, shows that $\nu_\infty^\alpha$ is locally exponentially stable in $\Wasserstein_1$ if and only if
\begin{equation}\label{eq:OU-gradient-stability}
\lambda_\alpha:=1+Je^{-1/4}\sin\alpha=1+\alpha\tan\alpha>0.
\end{equation}
Here, the second equality follows from \eqref{eq:OU-fixed-point}, whose solutions satisfy $\cos\alpha\neq0$. The representation of the linearized semigroup $(\S_t)_{t\geq0}$ simplifies to
\begin{equation}\label{eq:S-Omega-representation-OU-cos}
\S_tg(y)=
\begin{cases}
\displaystyle \P_t^\alpha\phi(y)+\int_0^t\P_s^\alpha\phi(y)\,\Omega_{t-s}(\mathbf{1})\,\d s,
& \text{if }g=\phi,\\[1em]
\displaystyle \P_t^\alpha g(y)+J\int_0^t\Theta_{t-s}^g(\mathbf{1})\,\S_s\phi(y)\,\d s,
& \text{otherwise},
\end{cases}
\end{equation}
where $\phi(y):=\cos y$, $\mathbf{1}$ denotes the constant function equal to $1$, and $\Omega_t(\mathbf{1})$ is identified with its constant value.

\begin{proof}[Proof of Corollary~\ref{cor:OU-confinement}]
The coefficient assumption \ref{hyp:H1} holds for $b(x)=-x$ and $F(x,y)=J\cos y$. Assumption \ref{hyp:H2} holds because \eqref{eq:OU-fixed-point} ensures that the Gaussian law $\nu_\infty^\alpha=\NNN(\alpha,1/2)$ is invariant. Assumption \ref{hyp:H3} follows from Proposition~\ref{prop:stability-equivalence}, since condition \eqref{eq:OU-gradient-stability} implies the local exponential stability estimate \eqref{eq:non-linear-stability}. Under the initial assumptions of the corollary, Theorem~\ref{thm:main} applies and yields the stated confinement estimate.
\end{proof}

\section{Extension to a double-well model}\label{sec:double-well}

We consider the double-well model with linear interaction studied in \cite{Tug14}. Although \ref{hyp:H1} fails, the conclusion of Theorem~\ref{thm:main} still holds.

\paragraph{The model and its invariant laws}
Consider the particle system
\begin{equation}\label{eq:DW-particles-intro}
\d X_t^{i,N}=\parent{-V'(X_t^{i,N})+\frac{\theta}{N}\sum_{j=1}^N\parent{X_t^{j,N}-X_t^{i,N}}}\d t+\d B_t^i,
\end{equation}
where $V(x)=x^4/4-x^2/2$ and $F(x,y)=\theta(y-x)$, with $\theta>0$. Thus $b(x)=-V'(x)=x-x^3$ is not globally Lipschitz and $F$ is unbounded, so \ref{hyp:H1} is not verified.

The limiting McKean--Vlasov equation is
\begin{equation}\label{eq:DW-MV}
\d X_t=\parent{X_t-X_t^3+\theta\parent{\EE[X_t]-X_t}}\d t+\d B_t.
\end{equation}
For $m\in\RR$, the diffusion with interaction field $\alpha_m(x):=\theta(m-x)$ has the unique invariant law
\begin{equation}\label{eq:tug-equilibrium}
\nu_m(\d x)=\frac1{Z_m}\exp\parent{-\frac{x^4}{2}+(1-\theta)x^2+2\theta mx}\,\d x.
\end{equation}
The fixed-point characterization remains valid: $\nu_m$ is invariant for \eqref{eq:DW-MV} if and only if
\begin{equation}\label{eq:DW-fixed-point}
m=\int_\RR x\,\nu_m(\d x).
\end{equation}
For such an $m$, assumption \ref{hyp:H2} holds with $\alpha=\alpha_m$ and $\nu_\infty^\alpha=\nu_m$. 

\begin{prop}[Exponentially long confinement for the double-well model]\label{prop:tug-confinement}
Let $m$ satisfy \eqref{eq:DW-fixed-point} and
\begin{equation}\label{eq:DW-stability}
r_m:=2\theta\int_\RR(x-m)^2\,\nu_m(\d x)<1.
\end{equation}
Then \ref{hyp:H3} holds at $\nu_m$, and the conclusion of Theorem~\ref{thm:main} remains valid under \eqref{eq:initial-conditions} with $\nu_\infty^\alpha=\nu_m$.
\end{prop}

\begin{proof}
Assumption \ref{hyp:H2} follows from \eqref{eq:DW-fixed-point} and the characterization of invariant laws given above. Since \ref{hyp:H1} is not satisfied, three parts of the proof must be adapted: the proof of \eqref{eq:P-hessian}, which uses the boundedness of $\nabla(b+\alpha)$, and the Volterra construction and the moment estimate \eqref{eq:collective-lyapunov}, which use the boundedness of $F$. We will also verify \ref{hyp:H3} under \eqref{eq:DW-stability}, using \cite{MR26}.

These adaptations suffice to recover the confinement proof. The Volterra construction and first-variation arguments provide the semigroup and Duhamel properties of Proposition~\ref{prop:sgfpl}, together with the identification needed to verify \ref{hyp:H3}. Once \ref{hyp:H3} is established, the diffusion estimates and \eqref{eq:K-properties} allow us to apply Lemmas~\ref{lem:terminal-time} and~\ref{lem:hessian} with their previous proofs; in particular, $\LLL_\alpha g$ and its gradient are bounded for $g\in\DDD$, since $\nabla g$ has compact support. The collective bound \eqref{eq:collective-lyapunov} gives \eqref{eq:V-ineq}, hence Lemma~\ref{lem:moment} and the moment control required for the finite-dimensional reduction. The Itô identities will be justified by localization. Finally, \eqref{eq:lem-ineg-on-G}--\eqref{eq:lem-ineg-on-G2} only use the bounded derivatives $\nabla_yF$ and $\nabla_x\nabla_yF$, so Lemma~\ref{lem:mild-estimates} and the final continuity argument will apply unchanged.

\paragraph{Preliminary estimates \eqref{eq:Minfty}--\eqref{eq:sensitivity-formula}}
The moment condition \eqref{eq:Minfty} is verified directly from the explicit density \eqref{eq:tug-equilibrium}.

The gradient estimate \eqref{eq:P-gradient} is obtained from \cite[Corollary~2]{Ebe16}, since the frozen drift $B:=b+\alpha$ is dissipative at sufficiently large distances. The required non-explosion and finite-time bounds on the stopped second moments of the diffusion \eqref{eq:mvll} are ensured by the Lyapunov estimate \eqref{eq:frozen-Lyapunov}, which holds with $\kappa=1$ and $C=C_m$.

The regularization estimates \eqref{eq:BEL-bound} and \eqref{eq:P-hessian} are recovered from \cite[Theorem~2.1 and the proof of Theorem~2.3]{EL94} by localization and approximation. Indeed, with $\ell:=\max\{1-\theta,0\}$, the inequalities $\nabla B\leq\ell$ and $\nabla B+|\nabla^2B|\leq\ell+3$ imply, through the variation equations, the required bounds on the first and second spatial variations, uniform in the initial point over finite time intervals.

The Wasserstein estimate \eqref{eq:controle-mf} is obtained by synchronous coupling, with $C_T=e^T$.

For initial laws in $\PPP_2(\RR)$, the sensitivity formula \eqref{eq:sensitivity-formula} is recovered by the same Itô argument. After cancellation of the drift through the backward equation, taking expectations is justified by \eqref{eq:P-gradient} and \eqref{eq:controle-mf}, since $\Phi(\mu_s)-\alpha$ is constant in space.

\paragraph{The Volterra construction and first variation}
The formulas in \eqref{eq:S-Omega-representation} remain well defined. Indeed, the linear growth of $F$, the finite-time moment bounds for the frozen diffusion, and the second moment of $\nu_m$, together with the finite-time bounds on $\Omega$, give $\alpha\in\H$ and
\[\norme{p_t^y}{2}+\norme{q_t^y}{2}\leq C_T(1+|y|),\qquad 0\leq t\leq T.\]
The integrability of the defining formulas is then ensured by the finite-time bounds on $\Theta^g$. The preservation of $\XXX$ is obtained by the same differentiation argument, using \eqref{eq:P-gradient} and \eqref{eq:P-hessian}.

With the preliminary estimates established above, the proofs of Lemmas~\ref{lem:interaction-expansion} and~\ref{lem:Omega-S-identification} and Proposition~\ref{prop:sgfpl} apply for initial laws in $\PPP_2(\RR)$.

\paragraph{The collective bound \eqref{eq:collective-lyapunov}}
The collective bound \eqref{eq:collective-lyapunov} is obtained with $\kappa_V=1$ and $K_V=17/8$, since $2xb(x)+1\leq-x^2+17/8$ and
\[\frac2{N^2}\sum_{i,j=1}^Nx_iF(x_i,x_j)=-\frac{\theta}{N^2}\sum_{i,j=1}^N(x_i-x_j)^2\leq0.\]
The inequality \eqref{eq:V-ineq} is then obtained by applying this bound to the
drift in Itô's formula for $V_t^N$, with the same martingale bracket. The mild
equation~\eqref{eq:me} is recovered by the same Itô argument, using the backward equation established in Proposition~\ref{prop:sgfpl}.

\paragraph{Verification of \ref{hyp:H3} and end of the proof}
For $\mu\in\PPP_2(\RR)$ sufficiently close to $\nu_m$ in $\Wasserstein_2$, the following estimate is obtained from \cite[Theorem~8, estimate~(36), and Section~3.2.1]{MR26}:
\[\Wasserstein_2(\TTT_t(\mu),\nu_m)\leq C_Me^{-\lambda_Mt}\Wasserstein_2(\mu,\nu_m),\qquad t\geq0.\]
Indeed, the mean map $a\mapsto\int_\RR x\,\nu_a(\d x)$ has derivative $r_m<1$ at $a=m$ by \eqref{eq:DW-stability}. This condition is satisfied by both non-symmetric equilibria whenever they exist, as shown in \cite[Section~3.2.1]{MR26}. In their normalization, the interaction parameter is $\theta/2$ and the temperature is $\sigma^2=1/2$. 
The averaged derivative $\nu_m(\nabla\S_tg)$ is then controlled by translating the initial equilibrium. Let $Z$ have law $\nu_m$ and set $\mu_\varepsilon:=\Law(Z+\varepsilon)$. By \eqref{eq:nonlinear-first-order-expansion}, $\nu_m(\nabla\S_tg)$ is the derivative at $\varepsilon=0$ of $\TTT_t(\mu_\varepsilon)(g)$. Since $\Wasserstein_2(\mu_\varepsilon,\nu_m)=|\varepsilon|$, the bound
\[\vabs{\nu_m(\nabla\S_tg)}\leq C_Me^{-\lambda_Mt}\norme{\nabla g}{\infty}\]
is obtained by applying the preceding convergence estimate, dividing by $|\varepsilon|$, and letting $\varepsilon\to0$. 
Since $\nabla\KKK h=\theta\nu_m(\nabla h)$, the following bound is obtained by inserting the preceding estimate into the Duhamel formula \eqref{eq:S-duhamel} and applying \eqref{eq:P-gradient}:
\[\norme{\nabla\S_tg}{\infty}\leq C_P\parent{e^{-\lambda_Pt}+\theta C_M\int_0^t e^{-\lambda_P(t-s)}e^{-\lambda_Ms}\,\d s}\norme{\nabla g}{\infty}.\]
Assumption \ref{hyp:H3} is therefore verified for every $0<\lambda_S<\min\{\lambda_P,\lambda_M\}$.

The stated confinement estimate is therefore obtained by the same argument as in Section~\ref{sec:proof}, using the estimates established above.
\end{proof}

\begin{appendix}
\section{Estimates on stochastic convolutions}\label{app:convolution}

This appendix establishes the stochastic convolution estimate used twice in the proof of Theorem~\ref{thm:main}. The processes considered have the form
\[\frac1N\sum_{j=1}^N\int_0^tH_j(t,s)\cdot\d B_s^j.\]
Since the kernels depend on the terminal time $t$, these processes are not, in general, martingales. One therefore cannot directly apply a martingale maximal inequality to their supremum over $t$. However, for each fixed $t$, the process
\[r\lms\frac1N\sum_{j=1}^N\int_0^rH_j(t,s)\cdot\d B_s^j,\qquad 0\leq r\leq t,\]
is a martingale under the integrability assumptions below. We combine this observation with a control of the increments in the terminal time on dyadic grids to obtain the required maximal estimate.

\begin{lem}[Exponential bound for stochastic convolutions]\label{lem:stochastic-convolution}
Let $B^1,\ldots,B^N$ be independent $d$-dimensional Brownian motions. For every $1\leq j\leq N$, let $H_j(t,s)$ be an $\RR^d$-valued random kernel, jointly measurable on $\{0\leq s\leq t\}$ and predictable in $s$ for every fixed $t$. Consider
\[R_t^N:=\frac1N\sum_{j=1}^N\int_0^tH_j(t,s)\cdot\d B_s^j.\]
Assume that there exist deterministic constants $a,b>0$ and $\rho\in(0,1]$ such that, almost surely, for every $t\geq0$,
\begin{equation}\label{eq:fixed-kernel-condition}
\frac1N\sum_{j=1}^N\int_0^t|H_j(t,s)|^2\,\d s\leq a,
\end{equation}
and, whenever $0\leq u\leq t\leq u+1$,
\begin{equation}\label{eq:increment-kernel-condition}
\frac1N\sum_{j=1}^N\left[\int_0^u|H_j(t,s)-H_j(u,s)|^2\,\d s+\int_u^t|H_j(t,s)|^2\,\d s\right]\leq b|t-u|^{2\rho}.
\end{equation}
Then $R^N$ admits a continuous modification for which the following estimate holds: there exist constants $C_\rho,c_\rho>0$, depending only on $\rho$, such that, for every $T,x>0$,
\begin{equation}\label{eq:convolution-maximal}
\PP\parent{\sup_{0\leq t\leq T}|R_t^N|>x}\leq C_\rho(1+T)\exp\parent{-c_\rho\frac{Nx^2}{a+b}}.
\end{equation}
\end{lem}

We first establish the maximal estimate on intervals of length one, and then use a union bound to cover $[0,T]$. On each interval, we combine a fixed-time estimate with a bound on the supremum of the increments, obtained here using dyadic grids. The latter bound can also be derived from two forms of the Garsia--Rodemich--Rumsey inequality: either its power-moment form \cite[Corollary~A.2]{FV10}, retaining the $\sqrt p$ growth of the $L^p$ bounds and summing the exponential series, or its general form with $\Psi(u)=e^{u^2}-1$ \cite[Lemma~1.1]{HL13}.

\begin{proof}
We prove the estimate on each interval $I_k:=[k,k+1]$, for integers $k\geq0$, and then sum the resulting bounds.

We first establish the fixed-time estimate. Fix $t\geq0$. Freezing the terminal time in the kernel, the process
\[r\in[0,t]\lms M_r^{N,t}:=\frac1N\sum_{j=1}^N\int_0^rH_j(t,s)\cdot\d B_s^j\]
is a continuous square-integrable martingale, with $M_t^{N,t}=R_t^N$. By independence of the Brownian motions and \eqref{eq:fixed-kernel-condition},
\[\langle M^{N,t}\rangle_t=\frac1{N^2}\sum_{j=1}^N\int_0^t|H_j(t,s)|^2\,\d s\leq\frac{a}{N}.\]
Its exponential supermartingale gives $\EE(e^{\theta M_r^{N, t}}) \leq e^{\theta^2 \frac{a}{2N}}$. Chernoff's inequality, applied to $M^{N,t}$ and $-M^{N,t}$, therefore gives
\begin{equation}\label{eq:fixed-tail}
\PP\parent{|R_t^N|>x} = \PP\parent{|M_t^{N,t}| > x} \leq2\exp\parent{-\frac{Nx^2}{2a}},\qquad x>0.
\end{equation}
In particular, for every \(k\geq0\), 
\begin{equation}\label{eq:fixed-tail2}
\PP\parent{|R_k^N|>x}\leq2\exp\parent{-\frac{N}{a}\frac{x^2}{2}}.
\end{equation}

We next estimate the increments in the terminal time. Let $0\leq u<t\leq u+1$. Splitting the stochastic integral at time $u$ gives
\[R_t^N-R_u^N=\frac1N\sum_{j=1}^N\int_0^u\parent{H_j(t,s)-H_j(u,s)}\cdot\d B_s^j+\frac1N\sum_{j=1}^N\int_u^tH_j(t,s)\cdot\d B_s^j.\]
For fixed $u,t$, this is again the terminal value of a continuous martingale. By \eqref{eq:increment-kernel-condition}, its quadratic variation is bounded by $b|t-u|^{2\rho}/N$. Consequently,
\begin{equation}\label{eq:increment-tail}
\PP\parent{|R_t^N-R_u^N|>x}\leq2\exp\parent{-\frac{Nx^2}{2b|t-u|^{2\rho}}},\qquad x>0.
\end{equation}

For the rest of the proof, we need a continuous modification of $R^N$. We therefore apply Kolmogorov's continuity criterion on each interval $I_k$. To this end, integrating \eqref{eq:increment-tail} yields, for every $p\geq1$ and $u,t\in I_k$,
\[\EE|R_t^N-R_u^N|^p=p\int_0^\infty x^{p-1}\PP\parent{|R_t^N-R_u^N|>x}\,\d x\leq C_p\parent{\frac bN}^{p/2}|t-u|^{\rho p}.\]
Choosing $p$ such that $\rho p>1$, the exponent of $|t-u|$ is strictly greater than $1$, so the criterion gives a continuous modification on each $I_k$. These modifications agree almost surely at their common endpoints. Since there are only countably many such endpoints, they yield a continuous modification on $[0,\infty)$. We work with this modification throughout the rest of the proof.

Fix $k\geq0$. Recall that we want to control $A_k:=\sup_{t\in I_k}|R_t^N|$. To do so, we observe that
\[A_k\leq|R_k^N|+B_k,\qquad\text{where}\quad B_k:=\sup_{t\in I_k}|R_t^N-R_k^N|.\]
Since the tail probability of $|R_k^N|$ is already controlled by \eqref{eq:fixed-tail2}, it remains to estimate $B_k$. We will then combine these two estimates using
\[\PP\parent{\sup_{t\in I_k}|R_t^N|>x}=\PP(A_k>x)\leq\PP\parent{|R_k^N|>\frac x2}+\PP\parent{B_k>\frac x2}.\]

To pass from the increment estimate at two fixed times \eqref{eq:increment-tail} to a supremum over $I_k$, we approximate every time by points on dyadic grids. For every $n\geq0$, divide $I_k$ into $2^n$ intervals of length $2^{-n}$ and define
\[p_n(t):=k+2^{-n}\left\lfloor2^n(t-k)\right\rfloor,\qquad t\in[k,k+1).\]
Then $p_0(t)=k$, $p_n(t)\to t$ as $n\to\infty$, and $p_n(t)-p_{n-1}(t)$ is either $0$ or $2^{-n}$ for $n\geq1$. Thus, each successive approximation adds at most one increment from the corresponding grid. The telescoping identity gives
\[R_{p_n(t)}^N-R_k^N=\sum_{i=1}^{n}\parent{R_{p_i(t)}^N-R_{p_{i-1}(t)}^N}.\]
By continuity, the left-hand side converges to $R_t^N-R_k^N$ as $n\to\infty$. On the right-hand side, we bound each term indexed by $i$ in absolute value by the largest increment on this grid denoted by
\[\Delta_{k,n}:=\max_{0\leq j<2^n}\vabs{R_{k+(j+1)2^{-n}}^N-R_{k+j2^{-n}}^N},\qquad n\geq1.\]
Taking absolute values and letting $n\to\infty$ therefore gives
\[|R_t^N-R_k^N|\leq\sum_{n\geq1}\Delta_{k,n}.\]
The same inequality holds at $t=k+1$ by continuity. Taking the supremum over $I_k$ yields
\[B_k\leq\sum_{n\geq1}\Delta_{k,n}.\]
We now distribute the threshold $x$ between the different scales. Recall that, for two real numbers $a$ and $b$, if $a+b>x$, then either $a>x/2$ or $b>x/2$. More generally, for any sequence of positive weights $(w_n)_{n\geq1}$ satisfying $\sum_{n\geq1}w_n=1$, we have
\[\sum_{n\geq1}\Delta_{k,n}>x\quad\Longrightarrow\quad\exists n\geq1,\quad\Delta_{k,n}>w_nx.\]
Indeed, if every term were bounded by $w_nx$, their sum would be at most $x$. Consequently, a union bound gives
\[\PP(B_k>x)\leq\sum_{n\geq1}\PP(\Delta_{k,n}>w_nx).\]
At level $n$, there are $2^n$ increments, each over an interval of length $2^{-n}$. Applying \eqref{eq:increment-tail} to each increment and taking another union bound gives
\[\PP(\Delta_{k,n}>xw_n)\leq\sum_{j=0}^{2^n-1}\PP\parent{\vabs{R_{k+(j+1)2^{-n}}^N-R_{k+j2^{-n}}^N}>xw_n}\leq2^{n+1}\exp\parent{-\frac{Nx^2w_n^2}{2b\,2^{-2n\rho}}}.\]
We choose positive weights $w_n$ so that $\sum_{n\geq1}w_n=1$ and $\frac{1}{2} w_n^2\,2^{2n\rho} = \kappa_\rho\,2^{n\rho}$ for some constant $\kappa_\rho > 0$. These conditions give $w_n:=(2^{\rho/2}-1)2^{-n\rho/2}$, for $n\geq1$ and $\kappa_\rho:=(2^{\rho/2}-1)^2/2$. Thus,
\[\PP(B_k>x)\leq\sum_{n\geq1}2^{n+1}\exp\parent{-\kappa_\rho\frac{Nx^2}{b}2^{n\rho}}.\]
It remains to bound this series. For every $y\geq1$, we have $y2^{n\rho}\geq y/2+2^{n\rho}/2$, and therefore
\begin{equation}\label{eq:dyadic-series}
\sum_{n\geq1}2^{n+1}e^{-y2^{n\rho}}\leq e^{-y/2}\sum_{n\geq1}2^{n+1}e^{-2^{n\rho}/2}\leq C_\rho e^{-y/2}.
\end{equation}
The series on the right-hand side is finite since $\rho>0$. Applying \eqref{eq:dyadic-series} with $y=\kappa_\rho Nx^2/b$ gives the required bound whenever $y\geq1$. When $y<1$, we use
\[\PP(B_k>x)\leq1\leq e^{1/2}e^{-y/2},\]
so the same estimate holds after increasing the multiplicative constant. Thus, there exist constants $C_\rho,c_\rho>0$, depending only on $\rho$, such that
\begin{equation}\label{eq:local-sup-increment}
\PP(B_k>x)\leq C_\rho\exp\parent{-c_\rho\frac{Nx^2}{b}}.
\end{equation}

We can now estimate $A_k$. Since $A_k\leq|R_k^N|+B_k$, applying \eqref{eq:fixed-tail} and \eqref{eq:local-sup-increment} with threshold $x/2$ gives
\[\PP(A_k>x)\leq\PP\parent{|R_k^N|>\frac x2}+\PP\parent{B_k>\frac x2}\leq2\exp\parent{-\frac{Nx^2}{8a}}+C_\rho\exp\parent{-c_\rho\frac{Nx^2}{4b}}.\]
Since $a,b\leq a+b$, increasing $C_\rho$ and decreasing $c_\rho$ if necessary yields
\[\PP(A_k>x)\leq C_\rho\exp\parent{-c_\rho\frac{Nx^2}{a+b}}.\]

Finally, set $K_T:=\lceil T\rceil$. Since $T>0$, the interval $[0,T]$ is covered by $I_0,\ldots,I_{K_T-1}$, and $K_T\leq1+T$. A union bound therefore gives
\[\PP\parent{\sup_{0\leq t\leq T}|R_t^N|>x}\leq\sum_{k=0}^{K_T-1}\PP(A_k>x)\leq C_\rho(1+T)\exp\parent{-c_\rho\frac{Nx^2}{a+b}},\]
which proves \eqref{eq:convolution-maximal}.
\end{proof}

\section{Construction of a finite smooth net}\label{app:finite-net}
We construct a finite net on $B_K$, then extend and regularize its elements on $\RR^d$, preserving the Lipschitz bound and the normalization at the origin.

\begin{proof}[Proof of Lemma~\ref{lem:net}]
Fix $K,\delta>0$.
We first construct a finite family of continuous functions satisfying the second condition on $B_K$. To this end, we work on the compact ball $B_K$ and apply the Arzelà--Ascoli theorem. Consider
\[\FFF_K:=\left\{f\in C(B_K;\RR):f(0)=0,\ |f(x)-f(y)|\leq|x-y|\text{ for every }x,y\in B_K\right\}.\]
Every $f\in\FFF_K$ satisfies $|f(x)|\leq|x|\leq K$, so $\FFF_K$ is uniformly bounded. Moreover, all its elements are $1$-Lipschitz, which gives equicontinuity. Finally, the conditions $f(0)=0$ and $|f(x)-f(y)|\leq|x-y|$ are preserved under uniform convergence, so $\FFF_K$ is closed in $C(B_K)$. By the Arzelà--Ascoli theorem, $\FFF_K$ is therefore compact and admits a finite $\delta/2$-net $f_1,\ldots,f_m\in\FFF_K$.

We next extend these functions to compactly supported functions on $\RR^d$, while preserving their normalization and Lipschitz constant. This will allow us to construct functions in $\DDD$ by regularization. For every $k\in\{1,\ldots,m\}$, define
\[h_k(x):=\inf_{z\in B_K}\parent{f_k(z)+|x-z|}\,\wedge\,\operatorname{dist}(x,B_{2K}^c),\qquad x\in\RR^d.\]
Both functions appearing in the minimum are $1$-Lipschitz, and therefore $h_k$ is globally $1$-Lipschitz. Moreover, it extends $f_k$, i.e. $h_k=f_k$ on $B_K$. Indeed, if $x\in B_K$, then $\operatorname{dist}(x,B_{2K}^c)=2K-|x|\geq K\geq f_k(x)$, while the Lipschitz property of $f_k$ gives
\[f_k(x)=\inf_{z\in B_K}\parent{f_k(z)+|x-z|}=h_k(x).\]
Finally, $h_k$ is compactly supported because it vanishes on $B_{2K}^c$. Indeed, if $x\in B_{2K}^c$, then $\operatorname{dist}(x,B_{2K}^c)=0$ and
\[\fa z\in B_K, \quad f_k(z)+|x-z| \geq -|z|+|x|-|z| \geq 2K - 2K = 0.\]
Consequently, $h_k=0$ on $B_{2K}^c$. Thus, $h_k$ is globally $1$-Lipschitz, compactly supported, and satisfies $h_k(0)=0$.

Thirdly, to obtain the first condition, it remains to regularize these functions. Let $\rho_\varepsilon$ be a standard smooth mollifier supported in $B_\varepsilon$ and define
\[g_k:=\rho_\varepsilon\ast h_k-(\rho_\varepsilon\ast h_k)(0).\]
This preserves the second condition: indeed, $g_k(0)=0$, and convolution does not increase the Lipschitz constant, so that $\norme{\nabla g_k}{\infty}\leq1$. Moreover, since $h_k$ is compactly supported, $g_k=\rho_\varepsilon\ast h_k-(\rho_\varepsilon\ast h_k)(0)$ is constant outside a compact set. Thus $g_k$ is smooth and $\nabla g_k\in C_c^\infty(\RR^d;\RR^d)$, so $g_k\in\DDD$.

Finally, it remains to prove the last condition, concerning the approximation on $B_K$ of every globally $1$-Lipschitz function $g$ satisfying $g(0)=0$. Let $g$ be such a function. We use the triangle inequality
\[\norme{g-g_k}{L^\infty(B_K)}\leq\norme{g-f_k}{L^\infty(B_K)}+\norme{f_k-g_k}{L^\infty(B_K)}.\]
For the first term on the right-hand side, since the restriction of $g$ to $B_K$ belongs to $\mathcal F_K$, there exists $k\in\{1,\ldots,m\}$ such that
\[\norme{g-f_k}{L^\infty(B_K)}\leq\frac{\delta}{2}.\]
For the second term, we use the following identity, valid for every $x\in B_K$, which follows from the definition of $g_k$ and from the fact that $h_k=f_k$ on $B_K$ and $h_k(0)=f_k(0)=0$:
\begin{equation}\label{eq:sec-term-rhs}
f_k(x)-g_k(x)=\parent{h_k(x)-(\rho_\varepsilon\ast h_k)(x)}-\parent{h_k(0)-(\rho_\varepsilon\ast h_k)(0)}.
\end{equation}
Since $h_k$ is globally $1$-Lipschitz and $\rho_\varepsilon$ is supported in $B_\varepsilon$, for every $x\in\RR^d$,
\[\vabs{h_k(x)-(\rho_\varepsilon\ast h_k)(x)}=\vabs{\int_{\RR^d}\rho_\varepsilon(y)\parent{h_k(x)-h_k(x-y)}\,\d y}\leq\int_{\RR^d}\rho_\varepsilon(y)|y|\,\d y\leq\varepsilon.\]
Hence, $\norme{\rho_\varepsilon\ast h_k-h_k}{\infty}\leq\varepsilon$. Applying the triangle inequality to \eqref{eq:sec-term-rhs}, we obtain
\[\norme{f_k-g_k}{L^\infty(B_K)}\leq2\norme{\rho_\varepsilon\ast h_k-h_k}{\infty}\leq2\varepsilon.\]
Choosing $\varepsilon\leq\delta/4$ and combining the preceding estimates gives
\[\norme{g-g_k}{L^\infty(B_K)}\leq\frac{\delta}{2}+2\varepsilon\leq\delta.\]
This proves the third condition and completes the proof.
\end{proof}
\end{appendix}

\bibliography{bibliographie_metastability}
\bibliographystyle{alpha}

\end{document}